\documentclass[10pt,reqno]{amsart}
\usepackage[T1]{fontenc}
\usepackage[utf8]{inputenc}
\usepackage{lmodern}
\usepackage{amsmath,amssymb,mathtools}
\usepackage{microtype}
\usepackage{xcolor}
\usepackage{enumitem}
\usepackage{comment}
\usepackage[alphabetic]{amsrefs}
\usepackage{booktabs}
\usepackage[a4paper,margin=28mm]{geometry}
\usepackage[hidelinks]{hyperref}
\usepackage[capitalise,nameinlink]{cleveref}   
\makeatletter
\AtBeginDocument{%
  \@ifundefined{firstaid@cref@updatelabeldata}{}{%
    \gdef\cref@fix@ctr{}\gdef\cref@fix@type{}%
    \let\cref@fix@orig@optarg\refstepcounter@optarg
    \def\refstepcounter@optarg[#1]#2{%
      \gdef\cref@fix@ctr{#2}\gdef\cref@fix@type{#1}%
      \cref@fix@orig@optarg[#1]{#2}}%
    \let\cref@fix@orig@noarg\refstepcounter@noarg
    \def\refstepcounter@noarg#1{%
      \gdef\cref@fix@ctr{#1}\gdef\cref@fix@type{#1}%
      \cref@fix@orig@noarg{#1}}%
    \renewcommand\firstaid@cref@updatelabeldata[1]{%
      \cref@constructprefix{#1}{\cref@result}%
      \edef\cref@fix@tmp{#1}%
      \ifx\cref@fix@tmp\cref@fix@ctr
        \let\@tempa\cref@fix@type
      \else
        \let\@tempa\cref@fix@tmp
      \fi
      \ifx\@tempa\cref@fix@tmp
        \@ifundefined{cref@#1@alias}{}%
          {\edef\@tempa{\csname cref@#1@alias\endcsname}}%
      \fi
      \protected@xdef\cref@gcurrentlabel@temp{%
        [\@tempa][\arabic{#1}][\cref@result]%
        \csname p@#1\expandafter\endcsname\csname the#1\endcsname}%
      \aftergroup\firstaid@cref@smugglelabel}%
    \let\cref@fix@orig@appendix\appendix
    \def\appendix{\cref@fix@orig@appendix
      \crefalias{section}{appendix}\crefalias{subsection}{subappendix}}%
  }%
}
\makeatother
\allowdisplaybreaks[1]

\providecommand{\germ}{\mathfrak}

\numberwithin{equation}{section}

\theoremstyle{plain}
\newtheorem{theorem}{Theorem}[section]
\newtheorem{proposition}[theorem]{Proposition}
\newtheorem{lemma}[theorem]{Lemma}
\newtheorem{corollary}[theorem]{Corollary}

\theoremstyle{definition}
\newtheorem{definition}[theorem]{Definition}
\theoremstyle{remark}
\newtheorem{remark}[theorem]{Remark}

\newtheorem{observation}[theorem]{Observation}

\crefformat{enumi}{(#2#1#3)}
\Crefformat{enumi}{(#2#1#3)}
\crefrangeformat{enumi}{(#3#1#4)--(#5#2#6)}
\crefname{observation}{Observation}{Observations}
\Crefname{observation}{Observation}{Observations}

\newcommand{\Z}{\mathbb{Z}}      
\newcommand{\C}{\mathbb{C}}      

\DeclareMathOperator{\soc}{soc}
\DeclareMathOperator{\supp}{supp}
\DeclareMathOperator{\Hom}{Hom}
\DeclareMathOperator{\End}{End}
\DeclareMathOperator{\Ext}{Ext}
\DeclareMathOperator{\ext}{ext}
\DeclareMathOperator{\Comp}{Comp}
\DeclareMathOperator{\Irr}{Irr}

\DeclareMathOperator{\rank}{rank}
\DeclareMathOperator{\Mat}{Mat}

\DeclareMathOperator{\GL}{GL}
\DeclareMathOperator{\Aut}{Aut}

\newcommand{\Sh}{\mathsf S}
\newcommand{\qf}{\mathsf q}
\newcommand{\hgr}{\mathsf h}
\DeclareMathOperator{\tr}{tr}

\newcommand{\m}{\mathfrak m}
\newcommand{\n}{\mathfrak n}
\newcommand{\bd}{\boldsymbol d}
\newcommand{\be}{\boldsymbol e}
\newcommand{\one}{\mathbf 1}
\newcommand{\sq}{\square}

\newcommand{\ladp}{\mathfrak p}
\newcommand{\ladq}{\mathfrak q}
\newcommand{\dfct}{\mathfrak d}
\newcommand{\SOC}{\mathcal{SOC}}
\newcommand{\Umin}{\mathcal U}
\newcommand{\NV}{\Lambda}
\newcommand{\bpi}{\boldsymbol\pi}
\newcommand{\Mq}{\mathbf M}
\newcommand{\Zel}{Z}
\newcommand{\Rlead}{\mathcal R}
\newcommand{\incl}{\iota}
\newcommand{\rowmap}{\lambda}
\DeclareMathOperator{\diag}{diag}
\DeclareMathOperator{\pr}{pr}
\newcommand{\Epairs}{\mathcal E}
\newcommand{\Ilist}{\mathcal A}

\newcommand{\gd}{acceptable}

\newcommand{\Lgd}{$L$-acceptable}
\newcommand{\Rgd}{$R$-acceptable}
\newcommand{\extremal}{left-extremal}
\newcommand{\Extremal}{Left-extremal}

\newcommand{\scom}{\perp}
\newcommand{\nscom}{\not\perp}

\definecolor{claudecol}{rgb}{0.72,0.09,0.45}

\newcommand{\dist}{distinguished}
\newcommand{\Dist}{Distinguished}

\newcommand{\gdim}{\operatorname{\underline{dim}}}
\newcommand{\Cm}[1]{\mathbf C_{#1}}
\newcommand{\Orb}[1]{\overline{\mathcal O}_{#1}}
\newcommand{\dat}[1]{\partial_{#1}}
\newcommand{\Tmap}[3][]{T\if\relax\detokenize{#1}\relax\else^{#1}\fi_{#2\if\relax\detokenize{#3}\relax\else;#3\fi}}

\begin{document}

\sbox0{\tiny$\square$}\sbox0{\scriptsize$\square$}%
\sbox0{\footnotesize$\square$}\sbox0{\small$\square$}%

\title[Socle formulas for rigid components]
{Socle formulas for rigid irreducible components\\
of Lusztig's nilpotent varieties in type $A$}

\author{Erez Lapid}
\address{Department of Mathematics, Weizmann Institute of Science, Rehovot 7610001, Israel}
\email{erez.m.lapid@gmail.com}

\author{Alberto M\'{\i}nguez}
\address{Departamento de \'Algebra and Instituto de Matem\'aticas (IMUS), Universidad de Sevilla,
C/ Tarfia s/n, 41012 Sevilla (Spain)}
\email{minguez.alberto@gmail.com}

\thanks{The second named author was supported by the Principal Investigator project PAT4832423 of the
Austrian Science Fund (FWF) and the projects ATR2024-154613 funded by
MICIU/\allowbreak AEI/\allowbreak 10.13039/\allowbreak 501100011033 and PID2024-156912NB-I00 funded by
MICIU/\allowbreak AEI/\allowbreak 10.13039/\allowbreak 501100011033/\allowbreak FEDER}

\date{\today}

\begin{abstract}
Let $F$ be a non-archimedean local field.
For any irreducible component $C$ of a nilpotent variety of type $A$
let $\bpi(C)$ be the irreducible representation of a general linear group over $F$ attached to it. We show that if
the multisegment of $C$ is multiplicity free, $C$ is rigid if and only if $\bpi(C)$ is $\sq$-irreducible. In this case,
\[
 \bpi(C*D)=\soc\bigl(\bpi(C)\times\bpi(D)\bigr),\qquad
 \bpi(D*C)=\soc\bigl(\bpi(D)\times\bpi(C)\bigr)
\]
for every irreducible component $D$, where $*$ is the binary operation defined by Aizenbud and Lapid,
in accordance with a conjecture of Lapid and M\'inguez.

We also give a computer-assisted example in $A_7$ of a non-simple rigid module $z\ne0$
which does not admit a proper rigid submodule $y\ne0$ such that $\Ext_\Pi^1(z,z/y)=0$, answering negatively a question of
Aizenbud and Lapid.
\end{abstract}

\maketitle
\enlargethispage{2pt}

\setcounter{tocdepth}{1}
\tableofcontents

\section{Introduction}\label{sec:intro}

Let $Q$ be a quiver without loops.
The irreducible components of Lusztig's nilpotent
varieties $\NV(\bd)$, as we vary over all graded dimensions $\bd$ with finite support,
form a set $\Comp$ that parametrizes Lusztig's canonical basis of the
negative part of the quantized enveloping algebra attached to $Q$, as well as
the vertices of the crystal graph $B(\infty)$ \cites{MR1035415,MR1088333,MR1115118,MR1458969}.

The set $\Comp$ carries a binary operation $*$, introduced and studied in \cite{MR4847234}.
Viewing $\NV(\bd)$ as modules over the preprojective algebra $\Pi$ of $Q$,
the product $C_1*C_2$ is the closure of the set of extensions of $x_1$ by $x_2$,
where $(x_1,x_2)$ ranges over the open subset of $C_1\times C_2$ on which
$\dim\Ext^1_\Pi(x_1,x_2)$ is minimal. In the case $C_2$ consists of
the simple module at a vertex $i$ of $Q$, $C_1\mapsto C_1*C_2$ is the crystal operator
$\tilde f_i$ of $B(\infty)$.

From now on we take $Q$ to be the biinfinite equioriented quiver of type $A$.
Fix a non-archimedean local field $F$ and a supercuspidal representation $\varrho$ of $\GL_\ell(F)$
and label the vertices of $Q$ by $\varrho\lvert\det\rvert^i$, $i\in\Z$ with arrows $i\to i+1$,
so that the vertex set becomes the cuspidal line $\Z_\varrho=\{\varrho\lvert\det\rvert^i:i\in\Z\}$.
The set $\Comp$ is parametrized by multisegments in the sense of Zelevinsky.
By Zelevinsky's classification \cite{MR584084} multisegments also parametrize the set $\Irr_\varrho$ of
irreducible smooth complex representations of the groups $\GL_n(F)$, $n\ge0$, with cuspidal support
in $\Z_\varrho$.
Combining the two we obtain a bijection $C\mapsto\bpi(C)$ between $\Comp$ and $\Irr_\varrho$.

In \cite{MR4847251}, it was conjectured that $*$ is compatible with (normalized) parabolic induction $\times$.
More precisely, for all $C_1,C_2\in\Comp$
\begin{equation} \label{eq:LM}
\text{the representation $\bpi(C_1*C_2)$ is a subrepresentation of $\bpi(C_1)\times\bpi(C_2)$.}
\end{equation}
(See \cite{MR4847251}*{Conjecture 5.1}.)

An especially important class of irreducible components are those that contain an open orbit,
or equivalently, a \emph{rigid} module $x$, i.e., $\Ext^1_\Pi(x,x)=0$.
These components are called rigid.
On the representation-theoretic side, an irreducible $\pi$ is
\emph{$\sq$-irreducible} if $\pi\times\pi$ is irreducible.

Following \cite{MR2115084}*{Conjecture 5.3}, it was conjectured in \cite{MR4847251}*{Conjecture 3.4.C} that for every $C\in\Comp$,
\begin{equation}\label{eq:rigid-conj}
 C\ \text{is rigid} \iff\bpi(C)\ \text{is $\sq$-irreducible.}
\end{equation}

In the terminology of \cite{MR3866895}, we say that $C=\Cm{\m}$ is regular if no two segments of $\m$ share a (left or right) endpoint.
The equivalence \eqref{eq:rigid-conj} was verified in \cite{MR3866895} in the case that $C$ is regular,
in which a precise combinatorial characterization of rigidity was also given. 

In general, when $\pi_1$, $\pi_2$ are irreducible with at least one of them $\sq$-irreducible, the socle
$\soc(\pi_1\times\pi_2)$ is irreducible and occurs with multiplicity one in the Jordan--H\"older sequence of $\pi_1\times\pi_2$.
Thus, if $\bpi(C_1)$ or $\bpi(C_2)$ is $\sq$-irreducible, then \eqref{eq:LM} becomes
\begin{equation}\label{eq:socle-form}
 \SOC(C_1,C_2):\qquad\bpi(C_1*C_2)=\soc\bigl(\bpi(C_1)\times\bpi(C_2)\bigr).
\end{equation}
This is useful because $\bpi(C_1)\times\bpi(C_2)$ is irreducible if and only if both $\SOC(C_1,C_2)$ and $\SOC(C_2,C_1)$ hold.
The main result of \cite{MR4847251} is that if $C$ is regular and rigid, then $\SOC(C,D)$ and $\SOC(D,C)$ hold for every $D\in\Comp$.

A multiplicity free multisegment is one without repeated segments, and we
extend this terminology to irreducible components. Every regular
multisegment is multiplicity free. Our main result is the following.

\begin{theorem}[\cref{cor:main} and \cref{thm:mf-equivalence}]\label{thm:intro}
Let $C$ be a multiplicity free irreducible component. Then $C$ is rigid if
and only if $\bpi(C)$ is $\sq$-irreducible. In this case $\SOC(C,D)$ and
$\SOC(D,C)$ hold for every $D\in\Comp$.
\end{theorem}

Thus \eqref{eq:rigid-conj} holds for multiplicity free components, and \eqref{eq:LM} holds whenever $C_1$ or $C_2$ is multiplicity free and rigid.
This extends the results of \cites{MR3866895,MR4847251} from the regular case.
An algorithm for deciding rigidity in the multiplicity free case was given in \cite{LapidGreedy}. We will not use it here.

Let us say a few words about the proof.

A basic inductive procedure to prove $\sq$-irreducibility, which had already been used in the regular case, is based on the following fact \cite{MR4847251}*{Lemma 5.5}.
Let $\pi=\soc(\pi_1\times\pi_2)$ where $\pi_1,\pi_2$ are irreducible with $\pi_2$ $\sq$-irreducible, and assume that $\pi\times\pi_1$ is irreducible.
Then $\pi$ and $\pi_1$ are $\sq$-irreducible. An analogous property holds on the geometric side: if $0\to y\to z\to x\to0$ is an exact sequence of
$\Pi$-modules with $y$ rigid and $\Ext^1_\Pi(z,x)=0$, then $z$ and $x$ are rigid.

The main new idea is that this procedure propagates the property $\SOC(C,D)$ (\cref{prop:propagation}).
More precisely, let $C=A*B$, where $A$ is rigid and commutes with $C$.
Assume that $\bpi(A)$ and $\bpi(B)$ are $\sq$-irreducible, and that $\SOC(A,D)$, $\SOC(D,A)$ and $\SOC(D,B)$ hold for every $D\in\Comp$.
Then $\bpi(C)$ is $\sq$-irreducible and $\SOC(D,C)$ holds for every $D$. (For $\SOC(C,D)$ one would need a dual decomposition.)
For a multiplicity free multisegment such a factorization is obtained by deleting a segment $\Delta$
minimizing the sum of its endpoints. The forward direction of \cref{thm:intro} follows by considering both $C$ and its dual.


In turn, \cref{prop:propagation} is proved by reducing, under suitable conditions on $A$ and $C$, the condition $\SOC(D,C)$ to the case where 
$D$ commutes with $A$ -- see \Cref{cor:perpendicular}.

The converse part of \cref{thm:intro} is also proved by \extremal{} deletion. Let $\Delta$ be as before.
The crucial statement is that if $\pi=\Zel(\m)$ is $\sq$-irreducible then
$\tau=\Zel(\m-\Delta)$ is $\sq$-irreducible and $\pi\times\tau$ is irreducible (\cref{thm:realdelete}).
A key invariant is the defect $\dfct(\sigma,\tau)$, which is computed in terms of the order of vanishing at $s=0$ of the
composition of the standard intertwining operators between $\sigma\lvert\det\rvert^s\times\tau$ and $\tau\times\sigma\lvert\det\rvert^s$.
(It is the analogue of the invariant $\mathfrak d$ of Kang, Kashiwara, Kim and Oh for simple modules over quiver Hecke algebras, which is defined through renormalized $R$-matrices \cite{MR3758148}.)
Under suitable conditions $\dfct(\sigma,\tau)=0$ exactly when $\sigma\times\tau$ is irreducible.
A key property is that the defect drops strictly when passing to a socle (\cref{lem:drop}).
It then remains to construct an auxiliary representation $\sigma$ such that $\dfct(\sigma,\pi)=1$ and
$\soc(\sigma\times\pi)=\eta\times\tau$ for some irreducible representation $\eta$ (\cref{lem:ladders}).

In principle, the method of proof gives a finite deterministic algorithm
for computing $C*D$ for multiplicity free rigid $C$ and any $D$. While
theoretically significant, the algorithm is not particularly efficient, in
contrast to the regular rigid case where the recipe of \cite{MR4847251}*{Remark 8.3} is much faster.

It is natural to ask what the limits of the method are.
A proper rigid submodule $y\ne0$ of $z$ is called \dist{} if $\Ext^1_\Pi(z,z/y)=0$.
If every non-simple rigid $z\ne0$ admitted a \dist{} submodule, then we could show the forward direction of \eqref{eq:rigid-conj}
as well as the properties $\SOC(C,D)$ and $\SOC(D,C)$ for every rigid $C$.
Unfortunately, there exists an indecomposable rigid module in type $A_7$
with thirty interval occurrences which admits no \dist{} submodule (\cref{sm:main}).
This answers in the negative a question raised in \cite{MR4847234}*{Remark 8.5}.

The paper is organized as follows. \Cref{sec:prelim} collects notation and
the results of \cites{MR4847234,MR4847251} that we use. \Cref{sec:proof}
contains the propagation argument and the proof of the forward direction of \cref{thm:intro}.
\Cref{sec:converse} proves the converse part,
and \cref{sec:thirty} gives an example of a rigid module without a \dist{} submodule.

\section{Notation and preliminaries}\label{sec:prelim}

\subsection{Conventions}\label{subsec:conventions}
Throughout this paper we fix an algebraically closed field $K$. All
pertinent geometric objects (vector spaces, algebras, varieties,
morphisms, dimensions) are implicitly over $K$.
We will freely use the terminology and the results of \cite{MR4847234} and \cite{MR4847251}.

We take for $Q$ the quiver with vertex set $I=\Z$ and arrows $i\to i+1$.
We only consider finite-dimensional modules, so that
every construction takes place in a finite type $A$ subquiver, which we
may enlarge as necessary -- no statement depends on the choice. We write
$\Pi$ for the preprojective algebra of $Q$ and $\NV(\bd)$ for
Lusztig's nilpotent variety of $\Pi$-modules of a finitely supported dimension vector $\bd$,
acted on by $G_{\bd}=\prod_i\GL_{\bd_i}$. Let $\Comp$ be the disjoint union of the
sets of irreducible components of the $\NV(\bd)$, including the zero
component $0$. For $C\in\Comp$ with $C\subset\NV(\bd)$ we write $\lvert C\rvert=\sum_i\bd_i$.
For a $Q$-module or $\Pi$-module $M$ we write
$\gdim M=(\dim M_i)_{i\in I}$ for its graded dimension (its dimension
vector), and $\dim M=\sum_i\dim M_i$ for its dimension.

The Euler form
\[
 \langle\bd,\be\rangle_Q=\sum_i\bd_i\be_i-\sum_i\bd_i\be_{i+1},
\]
satisfies
\begin{equation}\label{eq:euler-Q}
 \dim\Hom_Q(L_1,L_2)-\dim\Ext^1_Q(L_1,L_2)=\langle\gdim L_1,\gdim L_2\rangle_Q
\end{equation}
for any $Q$-modules $L_1,L_2$.

A \emph{segment} is a set $[a,b]=\{a,a+1,\dots,b\}$ with $a\le b$ in $I$.
We write $\Mq([a,b])$ for the interval representation of $Q$, which is $K$
at the vertices $a,\dots,b$ and zero elsewhere, and whose nonzero arrow
maps are the identity of $K$. By Gabriel's theorem \cite{MR0332887}, the
$\Mq([a,b])$ are, up to isomorphism, the indecomposable $Q$-modules.
A \emph{multisegment} is an element
\[
 \m=\sum_{a\le b}m_{a,b}[a,b],\qquad m_{a,b}\in\Z_{\ge0},
\]
of the free abelian monoid generated by the segments. It is
\emph{multiplicity free} if $m_{a,b}\le1$ for all $a\le b$. We set
\[
 \Mq(\m)=\bigoplus_{a\le b}\Mq([a,b])^{\oplus m_{a,b}};
\]
thus $\m\mapsto \Mq(\m)$ is a bijection between multisegments and isomorphism
classes of $Q$-modules.
We shall use repeatedly the endpoint formulas
\begin{subequations}\label{eq:segment}
\begin{align}
 \dim\Hom_Q\bigl(\Mq([a,b]),\Mq([c,d])\bigr)
 &=\begin{cases}1,&c\le a\le d\le b,\\0,&\text{otherwise},\end{cases}
 \label{eq:segment-hom}\\
 \dim\Ext^1_Q\bigl(\Mq([a,b]),\Mq([c,d])\bigr)
 &=\begin{cases}1,&a<c\le b+1\le d,\\0,&\text{otherwise}.\end{cases}
 \label{eq:segment-ext}
\end{align}
In particular every interval module is a brick. For segments
$\Gamma=[c,d]$ and $\Delta=[a,b]$ we say that $\Gamma$ \emph{precedes}
$\Delta$, and write $\Gamma\prec\Delta$, if $c<a$, $d<b$ and $a\le d+1$. We write $\Sh\Delta=[a+1,b+1]$ for the shift of
$\Delta$. Then \eqref{eq:segment-hom} and
\eqref{eq:segment-ext} say that
\begin{align}
 \Ext^1_Q\bigl(\Mq(\Gamma),\Mq(\Delta)\bigr)\ne0
 &\iff\Gamma\prec\Delta,\label{eq:prec-ext}\\
 \Hom_Q\bigl(\Mq(\Delta),\Mq(\Gamma)\bigr)\ne0
 &\iff\Gamma\prec\Sh\Delta.\label{eq:prec-hom}
\end{align}
\end{subequations}

\begin{remark}[Interval summands of submodules]\label{lem:sub-intervals}
Let $N\subset M$ be $Q$-modules and let $d$ be a vertex.
\begin{enumerate}[label=\textup{(\roman*)},ref=\roman*]
\item\label{item:sub-count}
The number of interval summands of $N$ ending at $d$ is at most the number
of interval summands of $M$ ending at $d$.
\item\label{item:sub-start}
For every interval summand $[c,d]$ of $N$ there is an interval summand
$[a,d]$ of $M$ with $a\le c$.
\end{enumerate}
\end{remark}

Recall that $\Cm{\m}$, the closure of the conormal bundle to the
$G_{\bd}$-orbit of $\Mq(\m)$, sets up the bijection between multisegments and $\Comp$. 

Let 
\[
 (\bd,\be)=\langle\bd,\be\rangle_Q+\langle\be,\bd\rangle_Q,
\]
be the symmetrization of the Euler form.
We will constantly use the Hom--Ext identity of Crawley-Boevey
\cite{MR1781930}*{Lemma 1} and the Ext duality for $\Pi$, in the form
\cite{MR4847234}*{(2.7)--(2.9)}:
\begin{subequations}
\begin{align}
 \dim\Ext^1_\Pi(x,y)
 &=\dim\Hom_\Pi(x,y)+\dim\Hom_\Pi(y,x)-(\gdim x,\gdim y),\label{eq:cb}\\
 \Ext^1_\Pi(x,y)&\simeq\Ext^1_\Pi(y,x)^*.\label{eq:ext-duality}
\end{align}
In particular
\begin{equation}\label{eq:self-ext}
 \dim\Ext^1_\Pi(x,x)=2\dim\End_\Pi(x)-(\gdim x,\gdim x),
\end{equation}
\end{subequations}
and the function $(x,y)\mapsto\dim\Ext^1_\Pi(x,y)$ is upper
semicontinuous. For $X,Y\in\Comp$ we write, as in
\cite{MR4847234}*{\S\S2--3},
\[
 \ext^1_\Pi(X,Y)=\min_{x\in X,\,y\in Y}\dim\Ext^1_\Pi(x,y).
\]
Denote by $\Umin(X,Y)\subset X\times Y$ the locus where this minimum is attained.
It is a nonempty open subset of $X\times Y$.
We say that $X$ and $Y$ \emph{strongly commute}, and write $X\scom Y$, if $\ext^1_\Pi(X,Y)=0$. By \eqref{eq:ext-duality}
the relation $\scom$ is symmetric.

A module $x$ is \emph{rigid} if $\Ext^1_\Pi(x,x)=0$; a component is rigid if it contains a rigid module,
equivalently if it contains an open $G_{\bd}$-orbit, and we refer to a
point of that orbit as a rigid representative. For a rigid module $x$ we
write $\Orb{x}$ for the closure of its orbit, the rigid component of which $x$
is a rigid representative. Writing $\qf(\bd)=\langle\bd,\bd\rangle_Q$, so
that $(\bd,\bd)=2\qf(\bd)$, \eqref{eq:self-ext} shows that
\begin{equation}\label{eq:rigid-q}
 x\text{ is rigid}\iff\dim\End_\Pi(x)=\qf(\gdim x).
\end{equation}

For a $Q$-module $M$, the space of reverse arrow maps satisfying the
preprojective relations is canonically
\begin{equation}\label{eq:reverse-space}
 \Ext^1_Q(M,M)^*.
\end{equation}
See \cite{MR4847234}*{\S9} for more details. For a $\Pi$-module $z$ with
underlying $Q$-module $z_Q$ we write $\dat{z}\in\Ext^1_Q(z_Q,z_Q)^*$ for the
corresponding element, and call it the \emph{data} of $z$. For
$\Pi$-modules $z,z'$ we use the map of \cite{MR4847234}*{(9.1)}
\begin{equation}\label{eq:T-map}
 \Tmap{z}{z'}:\Hom_Q(z_Q,z'_Q)\longrightarrow\Ext^1_Q(z'_Q,z_Q)^*,\qquad
 \Tmap{z}{z'}(f)(\xi)=\dat{z}(f^*\xi)-\dat{z'}(f_*\xi),
\end{equation}
where $f^*\xi$ and $f_*\xi$ are the pullback and the pushforward of $\xi$
along $f$. By \cite{MR4847234}*{(9.3)}, $\Hom_\Pi(z,z')=\ker \Tmap{z}{z'}$. We
write $\Tmap{z}{}=\Tmap{z}{z}$.

The following structure will be used only in \cref{sec:thirty}.

For $t\in K^\times$, multiplying every reverse arrow by $t$ preserves the
preprojective relations, which are homogeneous of degree one in the
reverse arrows, and therefore defines an automorphism of $\Pi$. Twisting
by it gives an autoequivalence $z\mapsto z^{(t)}$ of the category of
$\Pi$-modules, which is the identity on morphisms (and in particular
preserves submodules). In terms of data, $z^{(t)}$ has underlying
$Q$-module $z_Q$ and data $t\,\dat{z}$. 
\begin{equation}\label{eq:rescale}
\text{If $z$ is rigid, then $z^{(t)}\simeq z$ for every $t\in K^\times$.}
\end{equation}
Indeed, the modules $z^{(t)}$ form an irreducible
curve in $\NV(\gdim z)$ through $z$, and they are rigid, since $z\mapsto z^{(t)}$ is an autoequivalence. 
Hence they all lie in the irreducible component of $z$, which has only one open orbit.

\subsection{The binary operation}\label{subsec:binary}
By definition $X*Y$ is the closure of the set of middle terms of the
exact sequences
\begin{equation}\label{eq:product}
 0\longrightarrow y\longrightarrow z\longrightarrow x\longrightarrow0,\qquad(x,y)\in\Umin(X,Y).
\end{equation}
By \cite{MR4847234}*{Theorem 3.1} $X*Y$ is an irreducible component.
Moreover, by \cite{MR4847234}*{Corollary 3.3}, if $X\scom Y$ then
\begin{subequations}
\begin{equation}\label{eq:strong-direct-sum}
 X*Y=Y*X=X\oplus Y,
\end{equation}
the closure of the direct sum locus.

Note that in general $X\nscom X$. If $X\scom X$, we set
\[
X^{\oplus r}=\overbrace{X\oplus\dots\oplus X}^r,\ \  r\ge0.
\]
We have
\begin{equation}\label{eq:rigid-powers}
 X^{\oplus r}*X^{\oplus s}=X^{\oplus(r+s)},\ \ r,s\ge0.
\end{equation}
\end{subequations}
If $X$ is rigid, in which case $X\scom X$, then $X^{\oplus r}$ is rigid:
if $x$ is a rigid representative for $X$, then $x^{\oplus r}$ is a rigid representative for $X^{\oplus r}$.

We caution that $X\scom X$ is possible even if $X$ is not rigid.
See \cite{MR4847234}*{Remark 4.3}, and \cite{MR4847251}*{Example 2.3}.

Following \cite{MR4847234}*{\S6} we say that $X$ and $Y$ \emph{weakly commute}
if $X*Y=Y*X$. Strong commutation implies weak commutation by
\eqref{eq:strong-direct-sum}, and the converse fails in general, as the
examples just cited show. However, by \cite{MR4847234}*{Proposition 6.1},
\begin{equation}\label{eq:weak-strong}
 \text{if $X$ or $Y$ is rigid, then $X*Y=Y*X$ if and only if $X\scom Y$,}
\end{equation}
and in that case we simply say that $X$ and $Y$ \emph{commute}. We use the shorter term whenever a
rigid component is present and there is no risk of confusion.

In general the operation $*$ is not associative. However, by \cite{MR4847234}*{Corollary 7.4},
\begin{subequations}\label{eq:assoc-all}
\begin{equation}\label{eq:associativity}
 (X*Y)*Z=X*(Y*Z)
\end{equation}
provided that $X\scom Y$ or $Y\scom Z$.
Combined with \eqref{eq:strong-direct-sum}, this allows a strongly
commuting pair to be interchanged inside a triple product: if $X\scom Y$
then
\begin{equation}\label{eq:swap-left}
 X*(Y*Z)=(X*Y)*Z=(Y*X)*Z=Y*(X*Z);
\end{equation}
similarly, if $Y\scom Z$ then
\begin{equation}\label{eq:swap-right}
 (X*Y)*Z=X*(Y*Z)=X*(Z*Y)=(X*Z)*Y.
\end{equation}
In each case the outer equalities are instances of
\eqref{eq:associativity}, for the pair $X,Y$ or $Y,Z$ taken in either
order, and the middle one is \eqref{eq:strong-direct-sum}.
\end{subequations}

Dualizing the vector spaces and the arrows, and reflecting the vertex set
($i\mapsto-i$), gives an exact contravariant autoequivalence $M\mapsto M^\vee$ of the
category of $Q$-modules (the reflection turns the dual quiver back into $Q$). It takes
$\Mq([a,b])$ to $\Mq([-b,-a])$, and hence $\Mq(\m)^\vee\simeq\Mq(\m^\vee)$, where for a
multisegment $\m=\sum_{a\le b}m_{a,b}[a,b]$ we put $\m^\vee=\sum_{a\le b}m_{a,b}[-b,-a]$.
Applying the same operation to the reverse arrows as well, we obtain an exact contravariant
autoequivalence $x\mapsto x^\vee$ of the category of $\Pi$-modules lifting it.
It reverses exact sequences, preserves rigidity, and induces an involution $C\mapsto C^\vee$ of $\Comp$ with
\begin{equation}\label{eq:geometric-duality}
 \Cm{\m}^\vee=\Cm{\m^\vee},\qquad (X*Y)^\vee=Y^\vee*X^\vee,\qquad
 \lvert C^\vee\rvert=\lvert C\rvert,
\end{equation}
which preserves rigidity as well as the number of distinct segments. Thus
every statement about $\Pi$-modules or components admits a dual counterpart.

The following upward propagation of rigidity will be key.

\begin{lemma}[\cite{MR4847234}*{Lemma 8.1}]\label{lem:ascending}
Let
\[
0\to y\to z\to x\to0
\]
be a short exact sequence of $\Pi$-modules. If $y$ is rigid and 
$\Ext^1_\Pi(z,x)=0$ then $x$ and $z$ are rigid. Dually, if
$x$ rigid and $\Ext^1_\Pi(y,z)=0$, then $y$ and $z$ are rigid.
\end{lemma}

On the level of components, \cite{MR4847234}*{Corollary 8.2} gives
\begin{equation}\label{eq:rigid-product}
 \text{if $C=C_1*C_2$ with $C_2$ rigid and $C\scom C_1$, then $C$ is rigid if and only if $C_1$ is.}
\end{equation}

\subsection{Representations}\label{subsec:reps}
Let $F$ be a non-archimedean local field. All representations of the groups
$\GL_n(F)$ occurring below are smooth and defined over $\C$. Recall that
normalized parabolic induction $\times$ is biexact, faithful and associative.
Here, for representations $\pi_i$ of $\GL_{n_i}(F)$, $i=1,2$,
$\pi_1\times\pi_2$ is the representation of $\GL_{n_1+n_2}(F)$ obtained by
normalized parabolic induction of $\pi_1\otimes\pi_2$ from the standard
parabolic subgroup with Levi part $\GL_{n_1}(F)\times\GL_{n_2}(F)$.

Thus $\times$, together with induction in stages, endows the direct sum over
$n\ge0$ of the categories of finite length representations of $\GL_n(F)$ with the
structure of a ring category over $\C$ in the sense of \cite{MR3242743}*{Definition 4.2.3},
whose unit object $\one$ is the unit representation of $\GL_0(F)=\{1\}$; its
Grothendieck group is a commutative ring, although $\times$ is not commutative.

For a representation $\pi$ of finite length we recall from
\cite{MR3573961}*{Definition 2.3} that $\pi$ is \emph{socle irreducible},
abbreviated SI, if $\soc(\pi)$ is irreducible and occurs with multiplicity one
in the Jordan--H\"older sequence of $\pi$; see also \cite{MR4169050}*{Definition 1}.
Clearly, cf.~\cite{MR3573961}*{Remark 2.4},
\begin{equation} \label{eq:si-sub}
\text{if $\pi$ is SI and $0\ne\pi'\subseteq\pi$, then $\pi'$ is SI and
$\soc(\pi')=\soc(\pi)$.}
\end{equation}

Following \cite{MR3866895} we call an irreducible representation $\pi$
\emph{$\sq$-irreducible} if $\pi\times\pi$ is irreducible.

The results about $\sq$-irreducible representations that we use are the
following \cite{MR4847251}*{Theorems 4.1.B and 4.1.D, Lemma 5.5}; see also
\cites{MR3314831,MR3866895}.

\begin{theorem}[\cite{MR4847251}*{Theorems 4.1.B and 4.1.D, Lemma 5.5}]\label{thm:sq-input}
Let $\sigma$ and $\tau$ be irreducible representations.
\begin{enumerate}[label=\textup{(R\arabic*)},ref=R\arabic*]
\item\label{item:si}
If at least one of $\sigma,\tau$ is $\sq$-irreducible, then
$\sigma\times\tau$ and $\tau\times\sigma$ are SI.
\item\label{item:irr}
Under the same hypothesis, $\sigma\times\tau$ is irreducible if and only if
$\sigma\times\tau\simeq\tau\times\sigma$ if and only if $\soc(\sigma\times\tau)\simeq\soc(\tau\times\sigma)$.
\item\label{item:product}
If $\sigma$ and $\tau$ are both $\sq$-irreducible and $\sigma\times\tau$ is
irreducible, then $\sigma\times\tau$ is $\sq$-irreducible.
\item\label{item:inj}
If $\sigma$ is $\sq$-irreducible, then each of $\soc(\tau\times\sigma)$ and
$\soc(\sigma\times\tau)$ determines $\tau$ up to isomorphism.
\item\label{item:companion}
If $\tau$ is $\sq$-irreducible, $\pi=\soc(\sigma\times\tau)$ and
$\pi\times\sigma$ is irreducible, then $\sigma$ and $\pi$ are
$\sq$-irreducible.
\end{enumerate}
\end{theorem}

We shall use the following representation-theoretic counterparts of
\eqref{eq:assoc-all}. Let $\pi_1,\pi_2,\pi_3$ be irreducible. By
\eqref{eq:si-sub}, if $\pi_1\times\pi_2\times\pi_3$ is SI then
\begin{subequations}\label{eq:socle-all}
\begin{equation}\label{eq:socle-assoc}
 \soc\bigl(\soc(\pi_1\times\pi_2)\times\pi_3\bigr)
 =\soc\bigl(\pi_1\times\soc(\pi_2\times\pi_3)\bigr)
 =\soc(\pi_1\times\pi_2\times\pi_3).
\end{equation}
If $\pi_1$ and $\pi_2$ are $\sq$-irreducible and $\pi_1\times\pi_2$ is
irreducible, then
\begin{equation}\label{eq:rep-swap-left}
 \soc\bigl(\pi_1\times\soc(\pi_2\times\pi_3)\bigr)
 =\soc(\pi_1\times\pi_2\times\pi_3)
 =\soc(\pi_2\times\pi_1\times\pi_3)
 =\soc\bigl(\pi_2\times\soc(\pi_1\times\pi_3)\bigr);
\end{equation}
similarly if $\pi_2$ and $\pi_3$ are $\sq$-irreducible and $\pi_2\times\pi_3$
is irreducible, then
\begin{equation}\label{eq:rep-swap-right}
 \soc\bigl(\soc(\pi_1\times\pi_2)\times\pi_3\bigr)
 =\soc(\pi_1\times\pi_2\times\pi_3)
 =\soc(\pi_1\times\pi_3\times\pi_2)
 =\soc\bigl(\soc(\pi_1\times\pi_3)\times\pi_2\bigr).
\end{equation}
\end{subequations}
Indeed, under the hypothesis of \eqref{eq:rep-swap-left} we have
$\pi_1\times\pi_2\simeq\pi_2\times\pi_1$ by \cref{item:irr}, and this
representation is $\sq$-irreducible by \cref{item:product}. Hence
$\pi_1\times\pi_2\times\pi_3\simeq\pi_2\times\pi_1\times\pi_3$ is SI by
\cref{item:si}; the outer equalities in \eqref{eq:rep-swap-left} are
\eqref{eq:socle-assoc} for the triples $\pi_1,\pi_2,\pi_3$ and
$\pi_2,\pi_1,\pi_3$, and the middle one is the isomorphism. The proof of
\eqref{eq:rep-swap-right} is the same, with
$\pi_2\times\pi_3\simeq\pi_3\times\pi_2$ in place of
$\pi_1\times\pi_2\simeq\pi_2\times\pi_1$. All the inner socles occurring
here are irreducible by \cref{item:si}.

Fix now a supercuspidal representation $\varrho$ of $\GL_\ell(F)$ and write
\begin{equation}\label{eq:line}
 \Z_\varrho=\{\varrho\lvert\det\rvert^i:i\in\Z\}
\end{equation}
for the corresponding cuspidal line. We consider the representations whose
cuspidal support lies in $\Z_\varrho$.

We label the vertex $i\in I$ of $Q$ by $\varrho\lvert\det\rvert^i$, which
identifies $I$ with $\Z_\varrho$. Under this identification a segment
$[a,b]\subseteq I$ becomes the set
\[
 [a,b]_\varrho=\{\varrho\lvert\det\rvert^a,\varrho\lvert\det\rvert^{a+1},\dots,
 \varrho\lvert\det\rvert^b\},
\]
which is precisely a segment in the sense of \cite{MR584084}*{\S3.1}, a
multisegment $\m$ becomes a multisegment $\m_\varrho$ in the sense of
\cite{MR584084}*{\S6}, and the relation $\prec$ becomes the relation
``precedes'' of \cite{MR584084}*{\S4.1}. Let $\Zel([a,b]_\varrho)$ be the
unique irreducible subrepresentation of
$\varrho\lvert\det\rvert^a\times\dots\times\varrho\lvert\det\rvert^b$, and
for $\m=\Delta_1+\dots+\Delta_k$, ordered so that $\Delta_i$ does not
precede $\Delta_j$ whenever $i<j$, let $\Zel(\m_\varrho)$ be the unique
irreducible subrepresentation of
$\Zel((\Delta_1)_\varrho)\times\dots\times \Zel((\Delta_k)_\varrho)$
\cite{MR584084}*{\S3.1 and Theorem 6.1}. Under Zelevinsky's
classification, the bijection of \cref{sec:intro} is
$\bpi_\varrho(\Cm{\m})=\Zel(\m_\varrho)$, with $\bpi_\varrho(0)=\one$.

The notions introduced below are defined with reference to $\varrho$ via the bijection $\bpi_\varrho$.
All our results hold for every $\varrho$, with proofs that are uniform in $\varrho$, and the only place where two cuspidal
lines interact is the duality \eqref{eq:pi-duality}, which exchanges
$\varrho$ and $\varrho^\vee$. We therefore suppress $\varrho$ from the
notation and write $[a,b]$, $\Zel(\m)$ and $\bpi$ for $[a,b]_\varrho$,
$\Zel(\m_\varrho)$ and $\bpi_\varrho$, except in \cref{subsec:duality}.

It will be convenient to introduce the following terminology.\footnote{In \cite{MR4847251}*{Definition 5.6}
a different terminology was used, without requiring $\sq$-irreducibility.}
\begin{definition}\label{def:good}
Let $C\in\Comp$ be such that $\bpi(C)$ is $\sq$-irreducible. We say that $C$ is
\begin{enumerate}[label=\textup{(\roman*)},ref=\roman*]
\item\label{item:Lgood}
\emph{\Lgd} if $\SOC(C,D)$ for every $D\in\Comp$;
\item\label{item:Rgood}
\emph{\Rgd} if $\SOC(D,C)$ for every $D\in\Comp$;
\item\label{item:good}
\emph{\gd} if it is both \Lgd{} and \Rgd{}.
\end{enumerate}
\end{definition}

Thus $\sq$-irreducibility of $\bpi(C)$ is part of the definition, and each of
the hypotheses ``\gd{}'', ``\Lgd{}'' and ``\Rgd{}'' below carries it. By \cref{item:si} the
induced representations occurring in \cref{def:good} are SI.

Recall finally that a component is \emph{basic} if it is $\Cm{[a,b]}$ or
$\Cm{[a,a]+[a+1,a+1]+\dots+[b,b]}$ for some $a\le b$. The corresponding
representations are $\Zel([a,b])$ and $L([a,b])$, the unique irreducible
subrepresentation of
$\varrho\lvert\det\rvert^b\times\dots\times\varrho\lvert\det\rvert^a$
\cite{MR584084}*{\S9.1}.
Basic components are rigid by \cite{MR4847234}*{\S4} and \gd{} by \cite{MR4847251}*{Corollary 5.10 and \S7.1}.

More generally, every regular rigid component is \gd{} \cite{MR4847251}*{Theorems 6.4 and 6.5}.
We shall only use this for ladders, that is, multisegments $[a_1,b_1]+\dots+[a_k,b_k]$ with
$a_1<\dots<a_k$ and $b_1<\dots<b_k$ \cite{MR3163355}: by \cite{MR4847251}*{\S2.5.1 and Theorem 6.5},
\begin{equation}\label{eq:regular}
 \text{if $\m$ is a ladder, then $\Cm{\m}$ is rigid and \gd{}.}
\end{equation}

The following is essentially \cite{MR4847251}*{Lemma 5.8}.

\begin{lemma}\label{lem:good-product}
Let $B$ be \gd{}.
\begin{subequations}
\begin{enumerate}[label=\textup{(\roman*)},ref=\roman*]
\item\label{item:good-commutation}
If $X\scom B$, then
\begin{equation}\label{eq:good-irreducible}
 \bpi(B)\times\bpi(X)=\bpi(X)\times\bpi(B)=\bpi(B*X)
\end{equation}
is irreducible. If moreover $X$ is \gd{}, then $B*X$ is \gd{}.
\item\label{item:good-power}
If $B$ is rigid, then $B^{\oplus r}$ is \gd{} and
\begin{equation}\label{eq:power}
 \bpi(B^{\oplus r})=\bpi(B)^{\times r}
\end{equation}
is irreducible and $\sq$-irreducible for every $r\ge0$.
\item\label{item:good-converse}
If $B$ is rigid and $\bpi(B)\times\bpi(X)$ is irreducible, then
$X\scom B$. Thus, for rigid $B$, $\bpi(B)\times\bpi(X)$ is irreducible if
and only if $X\scom B$.
\end{enumerate}
\end{subequations}
\end{lemma}

Indeed, if $X\scom B$ then $B*X=X*B=B\oplus X$ by
\eqref{eq:strong-direct-sum}, so $\SOC(B,X)$ and $\SOC(X,B)$ give
$\soc\bigl(\bpi(B)\times\bpi(X)\bigr)=\soc\bigl(\bpi(X)\times\bpi(B)\bigr)=\bpi(B*X)$,
and \cref{item:irr} gives \eqref{eq:good-irreducible}. If moreover $X$ is
\gd{}, then $\bpi(B*X)$ is $\sq$-irreducible by \cref{item:product}, and the
socle identities for $B*X$ follow from those for $B$ and $X$
using \eqref{eq:associativity} and \eqref{eq:socle-assoc}.
\Cref{item:good-power} follows from \cref{item:good-commutation} by induction on $r$.
For \cref{item:good-converse}, if $\bpi(B)\times\bpi(X)$ is
irreducible, it is isomorphic to $\bpi(X)\times\bpi(B)$, so the socle
identities for $B$ give $\bpi(B*X)=\bpi(X*B)$; hence $B*X=X*B$, and since $B$
is rigid, $B$ and $X$ strongly commute by \eqref{eq:weak-strong}.

\subsection{Duality}\label{subsec:duality}
Taking contragredients replaces $\Z_\varrho$ by $\Z_{\varrho^\vee}$, and
\begin{subequations}\label{eq:rep-duality}
\begin{equation}\label{eq:pi-duality}
 \bpi_\varrho(C)^\vee=\bpi_{\varrho^\vee}(C^\vee).
\end{equation}
Finally, for irreducible $\tau,\pi_1,\pi_2$ we have the reciprocity
\begin{equation}\label{eq:embedding-duality}
 \tau\hookrightarrow\pi_1\times\pi_2
 \iff
 \tau^\vee\hookrightarrow\pi_2^\vee\times\pi_1^\vee.
\end{equation}
\end{subequations}
See \cite{MR4847251}*{(3.4)}, which goes back to \cite{MR0333080}*{Theorem 1}.

Therefore, if $\bpi(C)$ is $\sq$-irreducible then
\begin{equation}\label{eq:L-R}
 C\ \text{is \Lgd{}}\iff C^\vee\ \text{is \Rgd{}},
\end{equation}
where \Lgd{} refers to $\varrho$ and \Rgd{} to $\varrho^\vee$.

\section{Propagation}\label{sec:proof}

\subsection{Stable commutation}\label{subsec:coextension}

For a module $M$ over a ring $R$ and $\xi_1,\dots,\xi_r\in M$ we write
$R\langle\xi_1,\dots,\xi_r\rangle$ for the submodule of $M$ generated by
$\xi_1,\dots,\xi_r$, and if $M$ is finitely generated we write $\mu_R(M)$
for the minimal number of generators of $M$.
The following simple fact will play a key role.

\begin{observation}\label{lem:coextension}
Let $y$ be a rigid $\Pi$-module, let $x$ be any $\Pi$-module, and regard
$E=\Ext^1_\Pi(x,y)$ as a left module over $R=\End_\Pi(y)$, an endomorphism
acting by pushforward. Let $r\ge0$, $\xi_1,\dots,\xi_r\in E$, and let
\begin{equation}\label{eq:universal-extension}
 0\longrightarrow y^{\oplus r}\longrightarrow z\longrightarrow x\longrightarrow0
\end{equation}
represent $(\xi_1,\dots,\xi_r)$ under
$\Ext^1_\Pi(x,y^{\oplus r})\simeq E^{\oplus r}$. Then,
\begin{equation}\label{eq:ext-single}
 \Ext^1_\Pi(z,y)\simeq E/R\langle\xi_1,\dots,\xi_r\rangle.
\end{equation}
Indeed, applying $\Hom_\Pi(-,y)$ to \eqref{eq:universal-extension} gives the exact sequence
\[
\Hom_\Pi(y^{\oplus r},y)\to E\to\Ext^1_\Pi(z,y)\to\Ext^1_\Pi(y^{\oplus r},y)=0,
\]
whose first map sends $(\phi_1,\dots,\phi_r)\in R^{\oplus r}$ to $\sum_i(\phi_i)_*\xi_i$. 
\end{observation}

Now let $A\in\Comp$ be rigid with rigid representative $y$, let $R=\End_\Pi(y)$, let $D\in\Comp$, and put
$e=\ext^1_\Pi(D,A)$. Let
\begin{subequations}
\begin{equation}\label{eq:generic-locus}
 \mathcal X=\{x\in D:\dim\Ext^1_\Pi(x,y)=e\},
\end{equation}
a nonempty open subset of $D$, and for $x\in\mathcal X$ write
$E_x=\Ext^1_\Pi(x,y)$ for the $R$-module of \cref{lem:coextension}.
In general, if $D$ is not rigid, the module $E_x$ genuinely depends on $x$.

Put
\begin{equation}\label{eq:m-r}
 m_r=\max\bigl\{\dim_KR\langle\xi_1,\dots,\xi_r\rangle
 :x\in\mathcal X,\ \xi_1,\dots,\xi_r\in E_x\bigr\},
\end{equation}
\end{subequations}
and let $\mu(D,A)=\min_{x\in\mathcal X}\mu_R(E_x)\le e$, with
equality when $y$ is a brick, that is, when $R=K$.
Clearly
\begin{equation}\label{eq:m-r-monotone}
 0=m_0<m_1<\dots<m_{\mu(D,A)}=e=m_{\mu(D,A)+1}=\cdots.
\end{equation}

\begin{lemma}\label{cor:ext-strict}
With the assumptions and notation above, put $D_r=D*A^{\oplus r}$. Then
\begin{equation}\label{eq:ext-formula}
 \ext^1_\Pi(D_r,A)=e-m_r\qquad\text{for every }r\ge0.
\end{equation}
Consequently, for every $r\ge0$:
\begin{enumerate}[label=\textup{(\roman*)},ref=\roman*]
\item\label{item:ext-strict}
If $\ext^1_\Pi(D_r,A)>0$ then $\ext^1_\Pi(D_{r+1},A)<\ext^1_\Pi(D_r,A)$.
\item\label{item:ext-bound}
$\ext^1_\Pi(D_r,A)\le\max(e-r,0)$.
\item\label{item:commute}
$D_r\scom A$ if and only if $r\ge\mu(D,A)$.
\end{enumerate}
\end{lemma}

\begin{proof}
For $x\in\mathcal X$ we have $\dim\Ext^1_\Pi(x,y^{\oplus r})=re=\ext^1_\Pi(D,A^{\oplus r})$, so
$(x,y^{\oplus r})\in\Umin(D,A^{\oplus r})$ and the middle term $z$ of \eqref{eq:universal-extension} lies in $D_r$.
Conversely a pair of $\Umin(D,A^{\oplus r})$ whose second entry lies in the open orbit of
$A^{\oplus r}$ is of that form with $x\in\mathcal X$, and every extension of such an $x$ by $y^{\oplus r}$ is represented by a tuple in
$E_x^{\oplus r}$. Hence, as $x$ ranges over $\mathcal X$ and the $\xi_i$ over $E_x$, these $z$ run through a dense subset of $D_r$, and minimising
$\dim\Ext^1_\Pi(z,y)$, given by \eqref{eq:ext-single}, over that set gives \eqref{eq:ext-formula}. Assertions \cref{item:ext-strict},
\cref{item:ext-bound} and \cref{item:commute} follow from \eqref{eq:ext-formula} and \eqref{eq:m-r-monotone}.
\end{proof}

\subsection{Cancellation}\label{subsec:cancellation}

\Cref{cor:ext-strict} will be used in conjunction with the following result.

\begin{lemma}\label{lem:cancellation}
Let $B$ be \gd{}, let $C\scom B$ and suppose that $\bpi(C)$ is $\sq$-irreducible. Then for every $D\in\Comp$,
\[
 \SOC(D,C)\iff\SOC(D*B,C).
\]
\end{lemma}

\begin{proof}
Since $\bpi(C)$ is $\sq$-irreducible, the socles in the statement are irreducible by \cref{item:si}.
We show that the two sides of the second identity are obtained from the two sides of the first one by
applying $\tau\mapsto\soc\bigl(\tau\times\bpi(B)\bigr)$:
\begin{subequations}\label{eq:cancel}
\begin{align}
 \bpi\bigl((D*B)*C\bigr)&=\soc\bigl(\bpi(D*C)\times\bpi(B)\bigr),
 \label{eq:cancel-lhs}\\
 \soc\bigl(\bpi(D*B)\times\bpi(C)\bigr)&=
 \soc\Bigl(\soc\bigl(\bpi(D)\times\bpi(C)\bigr)\times\bpi(B)\Bigr).
 \label{eq:cancel-rhs}
\end{align}
\end{subequations}
Since $\bpi(B)$ is $\sq$-irreducible, $\tau\mapsto\soc\bigl(\tau\times\bpi(B)\bigr)$
is injective on irreducible representations by \cref{item:inj}, and the lemma follows.

For \eqref{eq:cancel-lhs}, since $B\scom C$, \eqref{eq:swap-right} gives $(D*B)*C=(D*C)*B$.
Since $B$ is \Rgd{}, $\SOC(D*C,B)$ holds, that is,
$\bpi\bigl((D*C)*B\bigr)=\soc\bigl(\bpi(D*C)\times\bpi(B)\bigr)$, which is \eqref{eq:cancel-lhs}.

For \eqref{eq:cancel-rhs}, since $B$ is \Rgd{} we have $\SOC(D,B)$, and $\bpi(B)\times\bpi(C)$
is irreducible by \cref{lem:good-product}\cref{item:good-commutation}.
Hence \eqref{eq:cancel-rhs} is \eqref{eq:rep-swap-right} for the triple $\bpi(D),\bpi(B),\bpi(C)$.
\end{proof}

\begin{corollary}\label{cor:perpendicular}
Let $A$ be rigid and \gd{}, and let $C\scom A$, with $\bpi(C)$ $\sq$-irreducible.
If $\SOC(D,C)$ for every $D\scom A$, then $C$ is \Rgd{}.
\end{corollary}

\begin{proof}
Given $D$, choose $r\ge\mu(D,A)$ and put $B=A^{\oplus r}$. Then
$D*B\scom A$ by \cref{cor:ext-strict}\cref{item:commute}, and $B$ is \gd{} by
\cref{lem:good-product}\cref{item:good-power}. Moreover $B\scom C$, since $A\scom C$.
By hypothesis, $\SOC(D*B,C)$ holds and \cref{lem:cancellation} then gives $\SOC(D,C)$.
\end{proof}

\subsection{One-sided propagation}\label{subsec:propagation}

\begin{proposition}[One-sided propagation]\label{prop:propagation}
Let $A$ be rigid and \gd{}, and let $B$ be \Rgd{}. Put $C=A*B$ and assume
that $A\scom C$. Then $C$ is \Rgd{}; in particular
$\bpi(C)$ is $\sq$-irreducible.
\end{proposition}

Note that we don't need to assume a priori that either $B$ or $C$ is rigid;
only the rigidity of $A$ is used, through \cref{cor:perpendicular}.
(Recall that the rigidity of $B$ implies that of $C$ by \eqref{eq:rigid-product}.)

\begin{proof}
Since $B$ is \Rgd{} (or since $A$ is \Lgd{}),
\begin{equation}\label{eq:rho-socle}
 \bpi(C)=\soc\bigl(\bpi(A)\times\bpi(B)\bigr).
\end{equation}
By \cref{lem:good-product}\cref{item:good-commutation}, applied to the
\gd{} component $A$ and $C\scom A$, the representation
$\bpi(C)\times\bpi(A)$ is irreducible. Hence \cref{item:companion} applied
to \eqref{eq:rho-socle} shows that $\bpi(C)$ is $\sq$-irreducible.

By \cref{cor:perpendicular} it remains to show $\SOC(D,C)$ for $D\scom A$. For such
$D$, $\bpi(D*A)=\bpi(D)\times\bpi(A)$ is irreducible by
\cref{lem:good-product}\cref{item:good-commutation}, so that
$\bpi(D)\times\bpi(A)\times\bpi(B)$ is SI by \cref{item:si}, and
\[
 \bpi(D*C)=\bpi\bigl((D*A)*B\bigr)
 =\soc\bigl(\bpi(D)\times\bpi(A)\times\bpi(B)\bigr)
 =\soc\bigl(\bpi(D)\times\bpi(C)\bigr).
\]
The first equality is \eqref{eq:associativity} for the pair $D\scom A$,
the second holds because $B$ is \Rgd{}, and the third is
\eqref{eq:socle-assoc} together with \eqref{eq:rho-socle}.
\end{proof}

\begin{remark}\label{rem:orders}
Note the asymmetry in \Cref{prop:propagation}.
A factorization $C=A*B$ shows that $C$ is \Rgd{}, because after reducing to $D\scom A$ the
expression $(D*A)*B$ may be reassociated, the pair $D,A$ being adjacent.
The expression $(A*B)*D$ has no such feature: the condition $A\scom D$ on
the outer factors is not a hypothesis under which
\eqref{eq:associativity} is available. The property \Lgd{} has to be obtained
by duality instead.
\end{remark}

\subsection{Propagation along factorizations}\label{subsec:conditional}

\begin{theorem}\label{thm:conditional}
Let $\mathcal S$ be a set of nonzero rigid components, closed under
$C\mapsto C^\vee$, such that every $C\in\mathcal S$ is either basic or of
the form $C=A*B$ with $A,B\in\mathcal S$ and $A\scom C$. Then every
member of $\mathcal S$ is \gd{}.
\end{theorem}

\begin{proof}
We argue by induction on $\lvert C\rvert$. If $C$ is basic, it is \gd{} (\cref{subsec:reps}).
Otherwise $C=A*B$ with $A,B\in\mathcal S$ and $A\scom C$. By induction hypothesis, $A$ and $B$ are \gd{}.
By \cref{prop:propagation}, $C$ is \Rgd{}.
Since $C^\vee\in\mathcal S$ is not basic either, and $\lvert C^\vee\rvert=\lvert C\rvert$, the same argument shows that
$C^\vee$ is \Rgd{}, that is, $C$ is \Lgd{} by \eqref{eq:L-R}. Hence $C$ is \gd{}.
\end{proof}

\begin{definition}\label{def:distinguished}
Let $z$ be a rigid $\Pi$-module. A proper, nonzero submodule $y$ of $z$ is \emph{\dist{}} if it is rigid and $\Ext^1_\Pi(z,z/y)=0$.
\end{definition}

In this case $x=z/y$ is rigid by \cref{lem:ascending}, and
\begin{equation}\label{eq:distinguished}
 \Orb{z}=\Orb{x}*\Orb{y},\qquad \Orb{z}\scom\Orb{x}.
\end{equation}
Dually, by \eqref{eq:geometric-duality}, $z^\vee$ has a \dist{} submodule if and
only if $z$ has a proper, nonzero submodule $y$ with $z/y$ rigid and $\Ext^1_\Pi(z,y)=0$.
Conversely, if $\Orb{z}=A*B$ with $A,B$ nonzero rigid components and $A\scom\Orb{z}$, then $z$ admits a \dist{} submodule.

Our original hope was that every nonzero non-simple rigid module admits a \dist{} submodule.
(This is the question raised in \cite{MR4847234}*{Remark 8.5}.)
\Cref{thm:conditional} would then imply that every rigid component is \gd{}.
Unfortunately, as we shall see in \cref{sec:thirty}, there exists a rigid module without a \dist{} submodule.
Nonetheless, \Cref{thm:conditional} handles the multiplicity free case, as we shall now see.

\subsection{The multiplicity free case}\label{subsec:mf}

\begin{definition}[\Extremal{} segments]\label{def:extremal}
Let $\Delta$ be a segment of a multisegment $\m$.
We say that $\Delta$ is \emph{\extremal{}} in $\m$ if no segment of $\n=\m-\Delta$
precedes $\Delta$ or $\Sh\Delta$, that is, by \eqref{eq:prec-ext} and
\eqref{eq:prec-hom}, if
\begin{equation}\label{eq:deletion-vanishing}
 \Ext^1_Q\bigl(\Mq(\n),\Mq(\Delta)\bigr)=0\quad\text{and}\quad
 \Hom_Q\bigl(\Mq(\Delta),\Mq(\n)\bigr)=0.
\end{equation}
\end{definition}

This is condition (3b) in the definition of detachable segments in \cite{MR3866895}*{Definition 3.1}.
Of course, we could have similarly defined right-extremal, but we will not need it.

Since $\Delta\prec\Sh\Delta$, a \extremal{} segment occurs with multiplicity one in $\m$.
Conversely, a segment $\Delta=[a,b]$ of $\m$ that minimizes $a+b$ and occurs with multiplicity one in $\m$ is \extremal{}.
Note that a \extremal{} segment is not necessarily unique.

\begin{observation}[\Extremal{} deletion]\label{prop:deletion}
Let $\Delta$ be \extremal{} in $\m$, $\n=\m-\Delta$ and $A=\Cm{\n}$.
Then $\Cm{\m}=A*\Cm{\Delta}$. If moreover $\Cm{\m}$ is rigid, then
$A$ is rigid and $\Cm{\m}\scom A$.
\end{observation}

Indeed, the first assertion is \cite{MR4847234}*{Corollary 9.4}, since the
first vanishing in \eqref{eq:deletion-vanishing} is that of the target of
the map $\Tmap{y}{x}$ of \eqref{eq:T-map}, for $x\in A$ and $y\in\Cm{\Delta}$. For
the second, \cite{MR4847234}*{Lemma 9.5}, applied to a rigid
representative $z$ of $\Cm{\m}$ (whose underlying $Q$-module is $\Mq(\m)$),
gives an exact sequence $0\to y\to z\to x\to0$ with $x\in A$ and $y\in\Cm{\Delta}$ rigid and
$\Ext^1_\Pi(z,x)=0$, so that $y$ is a \dist{} submodule of $z$.

\begin{corollary}[Multiplicity free rigid components are \gd{}]\label{cor:main}
Let $\m$ be a multiplicity free multisegment and suppose that
$C=\Cm{\m}$ is rigid. Then $C$ is \gd{}, that is, $\bpi(C)$ is
$\sq$-irreducible and
\begin{equation}\label{eq:main}
 \SOC(C,D)\quad\text{and}\quad\SOC(D,C)
\end{equation}
for every $D\in\Comp$, and $\bpi(C)\times\bpi(D)$ and $\bpi(D)\times\bpi(C)$ have
irreducible socle occurring with multiplicity one.
\end{corollary}

Indeed, the set $\mathcal S$ of rigid components $\Cm{\m}$ with
$\m\ne0$ multiplicity free is closed under $C\mapsto C^\vee$, and every
non-basic $C=\Cm{\m}\in\mathcal S$ has a \extremal{} segment $\Delta$, e.g.\
one minimizing the sum of its endpoints, so that $C=\Cm{\m-\Delta}*\Cm{\Delta}$
with $\Cm{\m-\Delta}\in\mathcal S$ and $\Cm{\m-\Delta}\scom C$ by
\cref{prop:deletion}. Thus \cref{thm:conditional} applies. The last
assertion is \cref{item:si}.

\section{The multiplicity free converse}\label{sec:converse}

In this section we prove that if $\m$ is multiplicity free and $\Zel(\m)$ is
$\sq$-irreducible, then $\Cm{\m}$ is rigid. Together with \cref{cor:main}
this gives the equivalence of rigidity, $\sq$-irreducibility and
acceptability in the multiplicity free case (\cref{thm:mf-equivalence}).
The argument removes a \extremal{} segment, as in \cref{subsec:mf}, and
controls the commutation of $\Zel(\m)$ with $\Zel(\m-\Delta)$ by means of
standard intertwining operators and an auxiliary ladder representation.

\subsection{Intertwining operators}\label{subsec:intertwiners}
Write $\nu=\lvert\det\rvert$. For nonzero representations $\sigma,\tau$ of
finite length let
\[
 M_{\sigma,\tau}(s):\sigma\nu^s\times\tau\longrightarrow\tau\times\sigma\nu^s
\]
be the standard intertwining operator, let $r_{\sigma,\tau}\ge0$ be the
order of its pole at $s=0$, and let
$\Rlead_{\sigma,\tau}=\lim_{s\to0}s^{r_{\sigma,\tau}}M_{\sigma,\tau}(s)$, a
nonzero intertwining operator $\sigma\times\tau\to\tau\times\sigma$, as in \cite{MR3866895}*{\S2.2}.
We shall use the following standard properties, where $\sigma_1,\sigma_2,\sigma_3$ are nonzero
representations of finite length.
\begin{enumerate}[label=\textup{(I\arabic*)},ref=I\arabic*]
\item\label{item:r-sub}
$r_{\tau,\sigma_2}\le r_{\sigma_1,\sigma_2}$ for every nonzero subquotient $\tau$ of $\sigma_1$. If
$\tau\subset\sigma_1$ is a subrepresentation, the restriction of $\Rlead_{\sigma_1,\sigma_2}$ to
$\tau\times\sigma_2$ is $\Rlead_{\tau,\sigma_2}$ if $r_{\tau,\sigma_2}=r_{\sigma_1,\sigma_2}$, and $0$ otherwise
\cite{MR3866895}*{Lemma 2.3, parts 2 and 3}.
\item\label{item:r-prod}
$r_{\sigma_1\times\sigma_2,\sigma_3}\le r_{\sigma_1,\sigma_3}+r_{\sigma_2,\sigma_3}$, with equality if
$\sigma_3$ is irreducible, in which case
$\Rlead_{\sigma_1\times\sigma_2,\sigma_3}=(\Rlead_{\sigma_1,\sigma_3}\times\mathrm{id}_{\sigma_2})\circ(\mathrm{id}_{\sigma_1}\times\Rlead_{\sigma_2,\sigma_3})$
\cite{MR3866895}*{Lemma 2.3, part 5}.
\item\label{item:r-scalar}
If $\sigma_1$ and $\sigma_2$ are irreducible, then
$M_{\sigma_2\nu^s,\sigma_1\nu^s}(-s)\circ M_{\sigma_1,\sigma_2}(s)$ is a scalar \cite{MR3866895}*{Lemma 2.3, part 6}.
\item\label{item:r-socle}
If $\pi$ is $\sq$-irreducible, then $\Rlead_{\pi,\pi}$ is a nonzero scalar, and for every irreducible
$\sigma$ the image of $\Rlead_{\sigma,\pi}$, respectively $\Rlead_{\pi,\sigma}$, is irreducible and equal to
$\soc(\pi\times\sigma)$, which is also the cosocle of $\sigma\times\pi$, respectively
$\soc(\sigma\times\pi)$, the cosocle of $\pi\times\sigma$,
\cite{MR3866895}*{Corollaries 2.4 and 2.5}.
\end{enumerate}
Moreover, $r_{\sigma,\tau}=r_{\tau^\vee,\sigma^\vee}$ \cite{2508.13817}*{Theorem 4.1, part 5}, so that
\cref{item:r-sub} and \cref{item:r-prod} have analogues for the second argument.

Assume from now on that $\sigma,\tau$ are irreducible.
The composition $M_{\tau\nu^s,\sigma\nu^s}(-s)\circ M_{\sigma,\tau}(s)$ is a scalar
by \cref{item:r-scalar}, and we let $\alpha(\sigma,\tau)\in\Z$
be its order of vanishing at $s=0$ (negative for a pole). This invariant is easy to compute.
It depends only on the cuspidal supports of $\sigma$ and $\tau$ and it is given explicitly by
\begin{equation}\label{eq:alpha}
\alpha\bigl(\bpi(C),\bpi(D)\bigr)=-(\gdim C,\gdim D)
\end{equation}
for any $C,D\in\Comp$ \cite{2508.13817}*{Lemma 1.1, part 1}.

Following \cite{2508.13817}*{\S4} we put
\begin{equation}\label{eq:defect}
 \dfct(\sigma,\tau)=r_{\sigma,\tau}+r_{\tau,\sigma}+\alpha(\sigma,\tau),
\end{equation}
the \emph{defect} of the pair.
It is symmetric and non-negative \cite{2508.13817}*{\S4}. 
It is closely related to the invariant $\mathfrak d(M,N)$ considered by Kang, Kashiwara, Kim and Oh for simple modules over quiver Hecke (KLR)
algebras \cite{MR3758148}*{Definition 3.2.2}.

If $\sigma$ or $\tau$ is $\sq$-irreducible, then
\begin{equation}\label{eq:defect-zero}
\text{$\dfct(\sigma,\tau)=0$ if and only if $\sigma\times\tau$ is irreducible}
\end{equation}
\cite{2508.13817}*{Lemma 4.5}; cf.\ \cite{MR3758148}*{Lemma 3.2.3}.

It was conjectured in \cite{2508.13817}*{Conjecture 4} that for all multisegments $\m,\n$
\begin{subequations}
\begin{equation}\label{eq:droschl-hom}
 r_{\Zel(\m),\Zel(\n)}=\hom_\Pi(\Cm{\m},\Cm{\n})
\end{equation}
and consequently (by \eqref{eq:alpha} and \eqref{eq:cb})
\begin{equation}\label{eq:droschl}
\dfct\bigl(\Zel(\m),\Zel(\n)\bigr)=\ext^1_\Pi(\Cm{\n},\Cm{\m}).
\end{equation}
\end{subequations}
These relations were proved in \cite{2508.13817}*{Theorem 1} (using \cite{MR4847251}*{Theorem 6.4}) in the case where $\Cm{\m}$ or $\Cm{\n}$ is regular and rigid.

\begin{lemma}\label{lem:drop}
Let $\pi$ be $\sq$-irreducible and $\sigma$ irreducible with $\dfct(\pi,\sigma)>0$. Then
\begin{equation}\label{eq:drop}
 \dfct(\pi,\tau)<\dfct(\pi,\sigma)
\end{equation}
for every irreducible subquotient $\tau$ of $\sigma\times\pi$. In particular,
if $\dfct(\pi,\sigma)=1$, then $\pi\times\tau$ is irreducible for every such $\tau$.
\end{lemma}

\begin{proof}
By \eqref{eq:defect-zero}, $\sigma\times\pi$ is reducible.
By \cref{item:r-socle}, the image of $\Rlead_{\sigma,\pi}$ is
irreducible and is the cosocle of $\sigma\times\pi$, which is isomorphic to
$\soc(\pi\times\sigma)$; likewise for $\Rlead_{\pi,\sigma}$, with the two factors
interchanged. As $\sigma\times\pi$ and $\pi\times\sigma$ are SI by \cref{item:si}
and have the same composition factors, their cosocles occur in them with
multiplicity one.

Suppose first that $\tau$ is not isomorphic to the cosocle of
$\sigma\times\pi$. Then $\tau\simeq U/V$ with $V\subset U\subset\ker \Rlead_{\sigma,\pi}$.
Since $\Rlead_{\pi,\pi}$ is a nonzero scalar by \cref{item:r-socle}, \cref{item:r-prod} shows that
$r_{\sigma\times\pi,\pi}=r_{\sigma,\pi}+r_{\pi,\pi}$ and that $\Rlead_{\sigma\times\pi,\pi}$ is a nonzero
multiple of $\Rlead_{\sigma,\pi}\times\mathrm{id}_\pi$. In particular, $\Rlead_{\sigma\times\pi,\pi}$
vanishes on $U\times\pi$, and \cref{item:r-sub} gives
$r_{\tau,\pi}\le r_{U,\pi}<r_{\sigma\times\pi,\pi}=r_{\sigma,\pi}+r_{\pi,\pi}$.
On the other hand, by the analogues of \cref{item:r-sub} and \cref{item:r-prod} for the second argument,
$r_{\pi,\tau}\le r_{\pi,\sigma\times\pi}\le r_{\pi,\sigma}+r_{\pi,\pi}$. Adding,
and using $\alpha(\pi,\tau)=\alpha(\pi,\sigma)+\alpha(\pi,\pi)$ (by \eqref{eq:alpha}) and
$2r_{\pi,\pi}+\alpha(\pi,\pi)=\dfct(\pi,\pi)=0$ (by \eqref{eq:defect-zero}), we obtain
$\dfct(\pi,\tau)<\dfct(\pi,\sigma)+\dfct(\pi,\pi)=\dfct(\pi,\sigma)$.

If $\tau$ is the cosocle of $\sigma\times\pi$, that is, $\tau\simeq\soc(\pi\times\sigma)$,
then $\tau$ is not isomorphic to the cosocle of $\pi\times\sigma$: otherwise
the socle of $\pi\times\sigma$, which has multiplicity one, would be a direct
summand, and $\pi\times\sigma$ would be irreducible. So $\tau$ is a
subquotient of $\ker \Rlead_{\pi,\sigma}$, and the same argument, with the roles of the two factors interchanged, applies.
The last assertion follows from \eqref{eq:defect-zero}.
\end{proof}

For the record we state the following corollary, although we will not use it here.

\begin{corollary}\label{cor:defect-one}
Let $\pi,\sigma$ be $\sq$-irreducible with $\dfct(\pi,\sigma)=1$. Then each of
$\tau=\soc(\sigma\times\pi)$ and $\tau'=\soc(\pi\times\sigma)$ is
$\sq$-irreducible, and its products with $\pi$ and with $\sigma$ are irreducible.
\end{corollary}

\begin{proof}
\Cref{lem:drop}, applied with each of $\pi,\sigma$ in the role of $\pi$,
gives the irreducibility of the four products. Since $\tau\times\sigma$ is
irreducible, \cref{item:companion}, applied to $\tau=\soc(\sigma\times\pi)$,
shows that $\tau$ is $\sq$-irreducible, and likewise for $\tau'$, using $\tau'\times\pi$.
\end{proof}

\subsection{Auxiliary ladders and \extremal{} deletion}\label{subsec:ladders}

\begin{theorem}\label{thm:realdelete}
Let $\Delta$ be \extremal{} in $\m$ and $\n=\m-\Delta$. If $\Zel(\m)$ is
$\sq$-irreducible, then $\Zel(\m)\times \Zel(\n)$ is irreducible and $\Zel(\n)$ is $\sq$-irreducible.
\end{theorem}

The auxiliary representations are supplied by the following construction.

\begin{lemma}\label{lem:ladders}
Let $\Delta$ be \extremal{} in $\m$, and put $\n=\m-\Delta$, $\pi=\Zel(\m)$ and $\tau=\Zel(\n)$.
Then, there exists a ladder representation $\sigma$ such that $\dfct(\sigma,\pi)=1$ and
$\soc(\sigma\times\pi)=\eta\times\tau$ for some irreducible representation $\eta$.
\end{lemma}

No $\sq$-irreducibility of $\pi$ or $\tau$ is assumed here.

\begin{proof}
Write $\Delta=[a,b]$, $A=\Cm{\n}$ and $C=\Cm{\m}$, so that $C=A*\Cm{\Delta}$ by \cref{prop:deletion}.
Upon shifting the vertex set if necessary, we may assume without loss of generality that $\min\supp\m>0$.
Fix $k>\max\supp\m$ and let
\[
 \ladp=\sum_{i=0}^{k}[i,b+i],\ \  \ladq=[0,a-1]+\sum_{i=1}^{k}[i,b+i].
\]
These are ladders, hence $\Cm{\ladp}$ and $\Cm{\ladq}$ are rigid and \gd{} by \eqref{eq:regular}.
Put $\sigma=\Zel(\ladq)$ and $\eta=\Zel(\ladp)$. Then by \cite{MR4847251}*{Example 2.1}
$\Cm{\ladp}=\Cm{\ladq}*\Cm{\Delta}$ and hence $\eta=\soc(\sigma\times \Zel(\Delta))$ (see also \cite{MR3866895}*{Lemma 3.11}).
(In general, products of a ladder representation with an arbitrary irreducible representation were studied in
\cite{MR3573961}*{\S6.3}; in particular, \cite{MR3573961}*{Proposition 6.20} gives a combinatorial criterion for
the socle of such a product to be $\Zel(\m+\n)$, and hence for its irreducibility; cf.\ \cite{MR3866895}*{Theorem 3.13}.
The situation at hand is easier than the general case.)

Let $J=[0,b+k]$, let $\Pi_J$ be the preprojective algebra of the full subquiver of $Q$ on $J$,
regard $\Pi_J$-modules as $\Pi$-modules vanishing outside $J$, and let $e_k$ be the idempotent at the vertex $k$.
The rigid representative of $\Cm{\ladp}$ is the projective $\Pi_J$-module $p=\Pi_Je_k$, a lamina in the sense of Ringel
\cite{MR1734897}*{\S4}; its quotient $q$ by the submodule $x_\Delta\simeq \Mq(\Delta)$
(with zero reverse arrows) at the bottom of $p$ is the rigid representative of
$\Cm{\ladq}$ \cite{MR4847251}*{\S2.5.1}. For $w$ supported in
$\supp\m\subset J\setminus\{k\}$ we have $\Hom_\Pi(p,w)=e_kw=0$ and
$\Ext^1_\Pi(p,w)=0$, hence $\Ext^1_\Pi(q,w)\simeq\Hom_\Pi(x_\Delta,w)$. This
vanishes for generic $w\in A$ by \eqref{eq:deletion-vanishing}, and is
one-dimensional for generic $w\in C=A*\Cm{\Delta}$, such $w$ being an extension of a generic element of $A$ by $x_\Delta$. Thus
\begin{equation}\label{eq:extaux}
 \Cm{\ladp}\scom A,\qquad\Cm{\ladq}\scom A,\qquad\ext^1_\Pi(\Cm{\ladq},C)=1.
\end{equation}
Therefore, by \eqref{eq:swap-left}, applied to the pair $\Cm{\ladq}\scom A$, and \eqref{eq:strong-direct-sum},
\[
 \Cm{\ladq}*C=\Cm{\ladq}*(A*\Cm{\Delta})=A*(\Cm{\ladq}*\Cm{\Delta})=A*\Cm{\ladp}=A\oplus\Cm{\ladp}.
\]
Since $\Cm{\ladq}$ and $\Cm{\ladp}$ are \gd{},
$\soc(\sigma\times\pi)=\bpi(A\oplus\Cm{\ladp})=\eta\times\tau$ by
\cref{lem:good-product}\cref{item:good-commutation}.

Finally, since $\Cm{\ladq}$ is regular and rigid, we can apply \eqref{eq:droschl}, \eqref{eq:ext-duality} and
\eqref{eq:extaux} to conclude that $\dfct(\sigma,\pi)=\ext^1_\Pi(C,\Cm{\ladq})=\ext^1_\Pi(\Cm{\ladq},C)=1$.
\end{proof}

\begin{proof}[Proof of \cref{thm:realdelete}]
Put $\pi=\Zel(\m)$ and $\tau=\Zel(\n)$, and let $\sigma$ and $\eta$ be as
in \cref{lem:ladders}, so that $\kappa=\soc(\sigma\times\pi)=\eta\times\tau$.
If $\pi$ is $\sq$-irreducible, \cref{lem:drop} shows that $\pi\times\kappa$ is irreducible.
Hence, $\tau\times\pi$ is irreducible. Now
$\pi=\soc\bigl(\tau\times \Zel(\Delta)\bigr)$ with $\Zel(\Delta)$
$\sq$-irreducible, so \cref{item:companion} shows that $\tau$ is
$\sq$-irreducible.\footnote{The $\sq$-irreducibility of $\Zel(\n)$ also follows from
\cite{MR3866895}*{Lemma 3.2}, since a \extremal{} segment is detachable in
the sense of \cite{MR3866895}*{Definition 3.1}.}
\end{proof}

\begin{theorem}\label{thm:mf-equivalence}
For every multiplicity free multisegment $\m$ the following conditions are
equivalent:
\begin{enumerate}[label=\textup{(\alph*)},ref=\alph*]
\item\label{item:mf-rigid}
$\Cm{\m}$ is rigid;
\item\label{item:mf-sq}
$\Zel(\m)$ is $\sq$-irreducible;
\item\label{item:mf-gd}
$\Cm{\m}$ is \gd{}.
\end{enumerate}
\end{theorem}

\begin{proof}
\Cref{item:mf-rigid} implies \cref{item:mf-gd} by \cref{cor:main}, and
\cref{item:mf-gd} implies \cref{item:mf-sq} by definition. We show that
\cref{item:mf-sq} implies \cref{item:mf-rigid} by induction on the number
of segments of $\m$, the cases of at most one segment being clear. Choose
a \extremal{} segment $\Delta$ of $\m$, for instance one minimizing the sum
of its endpoints, and put $\n=\m-\Delta$. By
\cref{thm:realdelete}, $\Zel(\n)$ is $\sq$-irreducible and
$\Zel(\m)\times \Zel(\n)$ is irreducible. By the induction hypothesis $A=\Cm{\n}$ is
rigid, hence \gd{} by \cref{cor:main}, and \cref{lem:good-product}\cref{item:good-converse} gives
$\Cm{\m}\scom A$. Since $\Cm{\m}=A*\Cm{\Delta}$ with $A$ and $\Cm{\Delta}$ rigid,
$\Cm{\m}$ is rigid by \eqref{eq:rigid-product}.
\end{proof}

In particular \eqref{eq:rigid-conj} holds for multiplicity free
components, which is the multiplicity free case of
\cite{MR4847251}*{Conjecture 3.4.C}.

\appendix

\section{A rigid module without a \texorpdfstring{\dist{} submodule}{distinguished submodule}}\label{sec:thirty}

\subsection{The module \texorpdfstring{$X$}{X}}\label{sm:setup}
The purpose of this appendix is to answer in the negative the question,
raised in \cite{MR4847234}*{Remark~8.5}, whether every non-simple rigid
$\Pi$-module has a \dist{} submodule (\cref{def:distinguished}); see
\cref{subsec:conditional}.

As allowed by the conventions of \cref{sec:prelim}, we work in the full
subquiver of $Q$ on the vertices $1,\dots,7$ and with its preprojective
algebra $\Pi$. As in \cref{sec:prelim}, the multisegment $\m_x$ of a
$\Pi$-module $x$ is that of its underlying $Q$-module, and we call the
number of segments of $\m_x$, counted with multiplicity, the number of
\emph{occurrences} of $x$. We write $\langle\cdot,\cdot\rangle$ for
$\langle\cdot,\cdot\rangle_Q$. Set
\[
 \Delta_1=[1,2],\quad \Delta_2=[2,5],\quad \Delta_3=[3,4],\quad
 \Delta_4=[4,7],\quad \Delta_5=[5,6],
\]
where $[a,b]$ also stands for the interval module $\Mq([a,b])$ of
\cref{sec:prelim}.

Let $W=K^3$ be the space of row vectors, let $V=\Mat_3(K)_0=\{\zeta\in\Mat_3(K):\tr\zeta=0\}$
be the space of $3\times3$ matrices of trace zero, and define
\[
 \rowmap_i:V\longrightarrow W,\qquad
 \rowmap_i(\zeta)=\text{the $i$-th row of $\zeta$},
 \quad i=1,2,3.
\]
Each $\rowmap_i$ is surjective, so that $\dim\ker \rowmap_i=5$. Put $M_1=M_2=W$ and
$M_3=M_4=M_5=V$, and
\begin{equation}\label{sm:layers}
 \begin{gathered}
 X_Q=\bigoplus_{i=1}^5\Delta_i\otimes M_i=P_0\oplus P_1\oplus P_2,\\
 P_0=\Delta_1\otimes W,\qquad
 P_1=\Delta_2\otimes W\oplus\Delta_3\otimes V,\qquad
 P_2=\Delta_4\otimes V\oplus\Delta_5\otimes V.
 \end{gathered}
\end{equation}
We call $M_i$, placed at a vertex $v\in\Delta_i$, the \emph{$\Delta_i$-space}
at $v$.

The two interval types in $P_1$, and the two in $P_2$, are strictly nested:
$[3,4]\subset[2,5]$ and $[5,6]\subset[4,7]$. By \eqref{eq:segment} there are no nonzero $Q$-maps and no extensions
between two strictly nested intervals, and an interval module is a brick
without self-extensions. Hence each $P_g$ is a rigid $Q$-module, and
\begin{equation}\label{sm:layer-end}
 \End_Q(P_0)=\End(W),\qquad
 \End_Q(P_1)=\End(W)\oplus\End(V),\qquad
 \End_Q(P_2)=\End(V)\oplus\End(V),
\end{equation}
where $\End(M_i)$ acts on $\Delta_i\otimes M_i$ through the second factor.

For later use we note that for any $Q$-submodules $z_g\subset P_g$ and $z'_k\subset P_k$ we have
\begin{subequations}\label{sm:layer-vanishing}
\begin{align}
 \Hom_Q(z_g,z'_k)&=0\qquad\text{unless $k\le g\le k+1$},
 \label{sm:hom-vanish}\\
 \Ext^1_Q(z_g,z'_k)&=0\qquad\text{unless $g\le k\le g+1$}.
 \label{sm:ext-vanish}
\end{align}
\end{subequations}
This follows from \cref{lem:sub-intervals} and \eqref{eq:segment}.

Identifying $\Hom(M_i,M_j)^*$ with $\Hom(M_j,M_i)$ by the trace pairing, we have
\[
 \Ext^1_Q(X_Q,X_Q)^*
 =\bigoplus_{i,j=1}^5\Ext^1_Q(\Delta_i,\Delta_j)^*\otimes\Hom(M_j,M_i).
\]
By \eqref{eq:segment-ext}, $\Ext^1_Q(\Delta_i,\Delta_j)$ is nonzero exactly for
$(i,j)$ in the set
\[
 \Epairs=\{(1,2),(1,3),(2,4),(2,5),(3,4),(3,5)\}
\]
of six pairs, and then it is one-dimensional. These are exactly the pairs
of interval types in $P_g$ and $P_{g+1}$, $g=0,1$, so that
\begin{equation}\label{sm:ext-decomp}
 \Ext^1_Q(X_Q,X_Q)^*=\Ext^1_Q(P_0,P_1)^*\oplus\Ext^1_Q(P_1,P_2)^*.
\end{equation}
For $(i,j)\in\Epairs$, an element of $\Ext^1_Q(\Delta_i,\Delta_j)^*$
amounts, by \eqref{eq:reverse-space}, to reverse arrow maps $v+1\to v$ from
$\Delta_j$ at $v+1$ to $\Delta_i$ at $v$, for the vertices $v$ with
$v\in\Delta_i$ and $v+1\in\Delta_j$, satisfying the preprojective
relations; these relations force all these maps to be multiplication by
one and the same scalar. We let $\varepsilon_{i,j}$ be the element for which
this scalar is $1$. 
Let $X$ be the $\Pi$-module with underlying $Q$-module $X_Q$ and data
\begin{equation}\label{sm:matrix}
 \dat{X}=\varepsilon_{1,2}\otimes1_W+\varepsilon_{1,3}\otimes \rowmap_1
 +\varepsilon_{2,4}\otimes \rowmap_2+\varepsilon_{2,5}\otimes \rowmap_3
 +\varepsilon_{3,4}\otimes1_V+\varepsilon_{3,5}\otimes1_V.
\end{equation}
Concretely, write $G_{i,j}\in\Hom(M_j,M_i)$ for the coefficient of
$\varepsilon_{i,j}$ in \eqref{sm:matrix}, so that $G_{1,2}=1_W$,
$G_{1,3}=\rowmap_1$, $G_{2,4}=\rowmap_2$, $G_{2,5}=\rowmap_3$ and $G_{3,4}=G_{3,5}=1_V$. At a
reverse arrow $v+1\to v$ the component of the reverse map of $X$ from the
$\Delta_j$-space at $v+1$ to the $\Delta_i$-space at $v$ is $G_{i,j}$ if
$(i,j)\in\Epairs$, $v\in\Delta_i$ and $v+1\in\Delta_j$, and all other
components vanish.

In particular, by \eqref{sm:ext-decomp}, the reverse maps of $X$ map
$P_{g+1}$ into $P_g$ for $g=0,1$ and vanish on $P_0$.

\begin{theorem}[Computer-assisted counterexample]\label{sm:main}
The $\Pi$-module $X$ is indecomposable and rigid, with
\[
 \m_X=3\Delta_1+3\Delta_2+8\Delta_3+8\Delta_4+8\Delta_5,
\]
so that $X$ has thirty occurrences. Moreover, $X$ has no \dist{} submodule: there
is no rigid submodule $0\ne N\subsetneq X$ such that
\[
 \Ext^1_\Pi(X,X/N)=0.
\]
\end{theorem}

Consequently the hypothesis of \cref{thm:conditional}
fails for the set of all nonzero rigid components (\cref{subsec:conditional}),
and the question of \cite{MR4847234}*{Remark 8.5} has a negative answer.

The proof occupies the rest of the appendix and proceeds as follows. We
first introduce a grading of $\Pi$-modules for which $X$ is graded, and
show that $\End_\Pi(X)=K1\oplus J$ with $J^2=0$; this gives
indecomposability and rigidity. Next, if $N$ is a \dist{} submodule of $X$, then $\Hom_\Pi(N,X)$ is a cyclic $\End_\Pi(X)$-module. We deduce
that $N$ is indecomposable and that, after an automorphism of $X$, $N$ is
a graded submodule of $X$. This gives two Euler identities and six linear
inequalities for the multiplicities of the intervals of $N$. Finally, a
finite enumeration shows that these conditions force $\m_N=\m_X$, that is,
$N=X$.

\subsection{Graded submodules of \texorpdfstring{$X$}{X}}\label{sm:layered-sec}
Grade the path algebra of the double quiver of $Q$ by giving the trivial
paths (the idempotents) and the arrows of $Q$ degree $0$, and the reverse
arrows degree $-1$.
The preprojective
relations are homogeneous, so $\Pi$ inherits a grading, which we call the
\emph{$\hgr$-grading}, to distinguish it from the grading by the vertices.

\begin{definition}\label{sm:layered-def}
An \emph{$\hgr$-graded} $\Pi$-module is a $\Pi$-module $z$ together with a
decomposition $z_Q=\bigoplus_{g\in\Z}z_g$ into $Q$-submodules such that the
reverse maps of $z$ map $z_{g+1}$ into $z_g$ for every $g$; equivalently,
$\dat{z}\in\bigoplus_g\Ext^1_Q(z_g,z_{g+1})^*$. For $\hgr$-graded $\Pi$-modules
$z,z'$ and $k\in\Z$ we put
\begin{equation}\label{sm:layer-diagonal}
 \Hom^k_\Pi(z,z')=\{f\in\Hom_\Pi(z,z'):f(z_g)\subset z'_{g+k}\text{ for all }g\},
\end{equation}
and write $\End^k_\Pi(z)=\Hom^k_\Pi(z,z)$. Define $\Hom^k_Q(z,z')$ in the same way, so that
\[
\Hom^k_Q(z,z')=\bigoplus_g\Hom_Q(z_g,z'_{g+k}).
\]
\end{definition}

By \eqref{sm:ext-decomp}, $X$ is $\hgr$-graded, with $X_g=P_g$ for $g=0,1,2$ and $X_g=0$ otherwise.

For $\hgr$-graded $\Pi$-modules $z,z'$ we write
\[
 \Tmap[0]{z}{z'}:\Hom^0_Q(z,z')\longrightarrow
 \Ext^1_Q(z'_Q,z_Q)^*
\]
for the restriction of the map $\Tmap{z}{z'}$ of \eqref{eq:T-map}, and
$\Tmap[0]{z}{}=\Tmap[0]{z}{z}$. Since $\Hom_\Pi(z,z')=\ker \Tmap{z}{z'}$, we have $\Hom^0_\Pi(z,z')=\ker \Tmap[0]{z}{z'}$.

\begin{lemma}\label{sm:layered}
Let $z,z'$ be $\hgr$-graded submodules of $X$. Then
\begin{subequations}
\begin{gather}
\label{sm:target}\Tmap[0]{z}{z'}\text{ takes values in }\Ext^1_Q(z'_0,z_1)^*\oplus\Ext^1_Q(z'_1,z_2)^*,\\
\label{sm:cross-maps} \Hom_\Pi(z,z')=\Hom^0_\Pi(z,z')\oplus\Hom^{-1}_Q(z,z').
\end{gather}
\end{subequations}
In particular, every $Q$-map of degree $-1$ from $z$ to $z'$ is a $\Pi$-map.
\end{lemma}

\begin{proof}
Since $\dat{z}$ and $\dat{z'}$ lie in the components of degree $(g,g+1)$, for
$f\in\Hom_Q(z_g,z'_{g+l})$ the functional $\Tmap{z}{z'}(f)$ can be nonzero only
on components $\Ext^1_Q(z'_h,z_k)$ with $k-h=1-l$. For $l=0$ this is
\eqref{sm:target}. For $l=-1$ the only such component is
$\Ext^1_Q(z'_0,z_2)$, which vanishes by \eqref{sm:layer-vanishing}, so
$\Tmap{z}{z'}(f)=0$. As $\Hom_Q(z_Q,z'_Q)=\Hom^0_Q(z,z')\oplus\Hom^{-1}_Q(z,z')$
by \eqref{sm:hom-vanish}, and $\Hom_\Pi=\ker T$, this proves \eqref{sm:cross-maps}.
\end{proof}

\begin{proposition}\label{sm:end}
The module $X$ is indecomposable and rigid, with
\[
 \gdim X=(3,6,11,19,19,16,8),\qquad\dim\End_\Pi(X)=\qf(\gdim X)=122.
\]
Moreover, $\End_\Pi(X)=K1\oplus J$ where
\[
J=\End^{-1}_\Pi(X)=\Hom_Q(P_1,P_0)\oplus\Hom_Q(P_2,P_1),\qquad J^2=0,\qquad \dim J=121.
\]
Finally, the map
\begin{equation}\label{sm:full}
 \Tmap[0]{X}{}:\bigoplus_{g=0}^2\End_Q(P_g)\longrightarrow
 \Ext^1_Q(P_0,P_1)^*\oplus\Ext^1_Q(P_1,P_2)^*
\end{equation}
is surjective.
\end{proposition}

\begin{proof}
We first determine $\End_\Pi(X)$. By \eqref{sm:cross-maps},
$\End_\Pi(X)=\End^0_\Pi(X)\oplus J$ with
$J=\End^{-1}_\Pi(X)=\Hom_Q(P_1,P_0)\oplus\Hom_Q(P_2,P_1)$. By
\eqref{eq:segment-hom}, the only nonzero $Q$-maps between interval types in
$P_g$ and in $P_{g-1}$ are
\[
 \Delta_2\longrightarrow \Delta_1,\qquad \Delta_4,\Delta_5\longrightarrow \Delta_2,\qquad \Delta_4\longrightarrow \Delta_3,
\]
so that $\dim J=9+24+24+64=121$. Products of two elements of $J$ have
degree $-2$, and so vanish by \eqref{sm:hom-vanish}.

It remains to show that $\End^0_\Pi(X)=\ker \Tmap[0]{X}{}$ is $K1$. By \eqref{sm:layer-end}, any element $u$
of $\bigoplus_g\End_Q(P_g)$ can be written as
\[
u=\bigoplus_i1_{\Delta_i}\otimes u_i\text{ with }u_i\in\End(M_i),
\]
and $u\in\ker \Tmap[0]{X}{}$ if and only if $u$ commutes with the
reverse maps of $X$. Since $u$ acts on the $\Delta_i$-spaces by $u_i$, and
the components of the reverse maps are the $G_{i,j}$, by the description
following \eqref{sm:matrix}, this means that
\[
 u_iG_{i,j}=G_{i,j}u_j\qquad\text{for }(i,j)\in\Epairs.
\]
The pair $(1,2)$ gives $u_1=u_2=:\phi\in\End(W)$, the pairs
$(3,4)$ and $(3,5)$ give $u_3=u_4=u_5=:\psi\in\End(V)$, and the remaining three pairs
give
\[
 \phi \rowmap_i=\rowmap_i\psi\qquad(i=1,2,3).
\]
The map $\rowmap=(\rowmap_1,\rowmap_2,\rowmap_3):V\to W^3$ is injective, with image the hyperplane
\[
 H=\{(w_1,w_2,w_3)\in W^3:e_1^*w_1+e_2^*w_2+e_3^*w_3=0\}.
\]
Identifying $V$ with $H$ by $\rowmap$, these three equations say that $\psi$ is
the restriction of $\diag(\phi,\phi,\phi)$ to $H$. In
particular $\diag(\phi,\phi,\phi)$ preserves $H$, so the functional
$(e_1^*,e_2^*,e_3^*)$ defining $H$ is an eigenvector of its transpose:
$e_i^*\phi=te_i^*$ for a scalar $t$ and $i=1,2,3$. Thus $\phi=t1$, and
$\psi=t1$.

Consequently, $\End_\Pi(X)=K1\oplus J$ is local, so
$X$ is indecomposable. The dimension vector follows from \eqref{sm:layers},
and a direct computation gives $\qf(\gdim X)=122=1+121$, so $X$ is rigid by
\eqref{eq:rigid-q}.

Finally, the source of $\Tmap[0]{X}{}$ has dimension
$9+(9+64)+(64+64)=210$ by \eqref{sm:layer-end}. The target is the sum of
the spaces $\Ext^1_Q(\Delta_i,\Delta_j)^*\otimes\Hom(M_j,M_i)$ over
$(i,j)\in\Epairs$, by \eqref{sm:ext-decomp}. Its dimension is $9+24+24+24+64+64=209$. As
$\ker \Tmap[0]{X}{}=K1$, the map $\Tmap[0]{X}{}$ is surjective.
\end{proof}

\subsection{\texorpdfstring{\Dist{} submodules}{Distinguished submodules}}\label{sm:cyclicity}
We say that a submodule $N$ of $X$, with inclusion $\incl$, is \emph{cyclic} if
\[
 \Hom_\Pi(N,X)=\End_\Pi(X)\,\incl.
\]
If $N$ is cyclic, then so is $\beta(N)$ for every $\beta\in\Aut_\Pi(X)$.
Since $\End_\Pi(X)=K1\oplus J$ is local by \cref{sm:end}, every
nonzero cyclic submodule of $X$ is indecomposable. Moreover, if $N$ and $N'$
are isomorphic nonzero cyclic submodules of $X$, then
\begin{equation}\label{sm:cyclic-unique}
 N'=(1+\theta)(N)\qquad\text{for some }\theta\in J.
\end{equation}
This is an instance of the uniqueness of left minimal versions of a
morphism \cite{MR1314422}*{Chapter I, \S2}.

Suppose, towards a contradiction, that $N$ is a \dist{} submodule of $X$, with
inclusion $\incl$, and put $Y=X/N$, so that $\Ext^1_\Pi(X,Y)=0$. Applying
$\Hom_\Pi(-,X)$ and using $\Ext^1_\Pi(Y,X)\simeq\Ext^1_\Pi(X,Y)^*=0$, by
\eqref{eq:ext-duality}, shows that $N$ is cyclic:
\begin{equation}\label{sm:cyclic}
 \Hom_\Pi(N,X)=\End_\Pi(X)\,\incl=K\incl\oplus J\incl,
\end{equation}
the second equality by \cref{sm:end}.

\begin{proposition}[Grading cyclic submodules]\label{sm:grading}
Let $N$ be a nonzero cyclic submodule of $X$ such that $N^{(t)}\simeq N$ for
all $t\in K^\times$; by \eqref{eq:rescale}, this holds if $N$ is rigid. Then
there is $\beta\in\Aut_\Pi(X)$ such that $\beta(N)$ is an $\hgr$-graded submodule
of $X$.
\end{proposition}

\begin{proof}
Fix $t\in K^\times$ with $t\ne\pm1$, so that $t^2$, $t$ and $1$ are pairwise
distinct, and let $\gamma=t^21_{P_0}\oplus t1_{P_1}\oplus1_{P_2}\in\Aut_Q(X_Q)$.
Since the reverse maps of $X$ lower the $\hgr$-degree by one, conjugation
by $\gamma$ multiplies each of them by $t$. In other words, $\gamma$ is an
isomorphism of $\Pi$-modules $X^{(t^{-1})}\to X$. As the twist
$z\mapsto z^{(t^{-1})}$ is the identity on morphisms, $X^{(t^{-1})}$ has
the same submodules and the same endomorphisms as $X$. Hence:
\begin{enumerate}[label=\textup{(\alph*)},ref=\alph*]
\item\label{item:gamma-sub}
$\gamma$ maps every submodule $N$ of $X$ onto a submodule $\gamma(N)$, and
restricts to an isomorphism $N^{(t^{-1})}\to\gamma(N)$;
\item\label{item:gamma-end}
$\phi\mapsto\gamma\phi\gamma^{-1}$ is an automorphism of $\End_\Pi(X)$;
as the elements of $J=\End^{-1}_\Pi(X)$ lower the $\hgr$-degree by one, it
is multiplication by $t$ on $J$.
\end{enumerate}
By transport of structure along $\gamma$, $\gamma(N)$ is again cyclic, and
$\gamma(N)\simeq N^{(t^{-1})}\simeq N$. By \eqref{sm:cyclic-unique},
$\gamma(N)=(1+\theta_0)(N)$ for some $\theta_0\in J$. Since $J^2=0$, for $\theta\in J$ we
have
\[
 \gamma\bigl((1+\theta)(N)\bigr)=(1+t\theta)\,\gamma(N)=(1+t\theta)(1+\theta_0)(N)=(1+t\theta+\theta_0)(N).
\]
For $\theta=(1-t)^{-1}\theta_0$ we have $t\theta+\theta_0=\theta$, so that $\beta=1+\theta\in\Aut_\Pi(X)$
satisfies $\gamma(\beta(N))=\beta(N)$. Since $\gamma$ acts on $P_0,P_1,P_2$ by the distinct
scalars $t^2,t,1$, it follows that $\beta(N)=\bigoplus_g\bigl(\beta(N)\cap P_g\bigr)$,
that is, $\beta(N)$ is $\hgr$-graded.
\end{proof}

\subsection{Numerical constraints}\label{sm:labels}
By \cref{sm:grading} we may, and do, assume that $N$ is an $\hgr$-graded
submodule of $X$, so that $N_g=N\cap P_g$. We write $\boldsymbol n_g=\gdim N_g$ and
$\boldsymbol p_g=\gdim P_g$ for the dimension vectors, so that by \eqref{sm:layers}
\[
 \boldsymbol p_0=(3,3,0,0,0,0,0),\qquad \boldsymbol p_1=(0,3,11,11,3,0,0),\qquad \boldsymbol p_2=(0,0,0,8,16,16,8).
\]

By \eqref{sm:cyclic}, $\Hom_\Pi(N,X)=K\incl\oplus J\incl$. Since
$J=\End^{-1}_\Pi(X)$, we have $J\incl\subset\Hom^{-1}_\Pi(N,X)$, so that
$\Hom^0_\Pi(N,X)=K\incl$ is one-dimensional. If $f\in\End^0_\Pi(N)$, then
$\incl f\in\Hom^0_\Pi(N,X)$, and $f\mapsto \incl f$ is injective. Hence
$\End^0_\Pi(N)=\ker \Tmap[0]{N}{}$ is also one-dimensional, spanned by $1_N$.

We have the following Euler identities for $N$.
\begin{lemma}\label{sm:euler}
Each $N_g$ is rigid as a $Q$-module.
Moreover,
\begin{subequations}
\begin{gather}
\label{sm:q0}
 \sum_{g=0}^2\qf(\boldsymbol n_g)+\langle \boldsymbol n_0,\boldsymbol n_1\rangle+\langle \boldsymbol n_1,\boldsymbol n_2\rangle=1,\\
 \label{sm:lin}
 \sum_{g=0}^2\langle \boldsymbol n_g,\boldsymbol p_g\rangle+\langle \boldsymbol p_0,\boldsymbol n_1\rangle+\langle \boldsymbol p_1,\boldsymbol n_2\rangle=1.
\end{gather}
\end{subequations}
\end{lemma}

\begin{proof}
We first prove \eqref{sm:q0}. By \eqref{sm:cross-maps}, and since
$\ker \Tmap[0]{N}{}$ is one-dimensional,
\[
 \dim\End_\Pi(N)=1+\dim\Hom_Q(N_1,N_0)+\dim\Hom_Q(N_2,N_1).
\]
On the other hand, by \eqref{eq:euler-Q}, \eqref{sm:hom-vanish} and \eqref{sm:ext-vanish},
\[
 \qf(\gdim N)=\sum_g\qf(\boldsymbol n_g)+\dim\Hom_Q(N_1,N_0)+\dim\Hom_Q(N_2,N_1)
 +\langle \boldsymbol n_0,\boldsymbol n_1\rangle+\langle \boldsymbol n_1,\boldsymbol n_2\rangle,
\]
and $\dim\End_\Pi(N)=\qf(\gdim N)$ by \eqref{eq:rigid-q}, since $N$ is
rigid. This proves \eqref{sm:q0}.

Next we show that each $N_g$ is rigid. By \eqref{sm:target}
the map $\Tmap[0]{N}{}$ takes values in $\Ext^1_Q(N_0,N_1)^*\oplus\Ext^1_Q(N_1,N_2)^*$.
Since $\ker \Tmap[0]{N}{}$ is one-dimensional, \eqref{eq:euler-Q} and
\eqref{sm:hom-vanish} give
\[
 \sum_g\dim\End_Q(N_g)-1=\rank \Tmap[0]{N}{}
 \le\sum_{g=0}^1\dim\Ext^1_Q(N_g,N_{g+1})
 =-\langle \boldsymbol n_0,\boldsymbol n_1\rangle-\langle \boldsymbol n_1,\boldsymbol n_2\rangle.
\]
By \eqref{sm:q0} and
$\qf(\boldsymbol n_g)=\dim\End_Q(N_g)-\dim\Ext^1_Q(N_g,N_g)$, the right-hand side is
$\sum_g\dim\End_Q(N_g)-\sum_g\dim\Ext^1_Q(N_g,N_g)-1$. Hence
$\sum_g\dim\Ext^1_Q(N_g,N_g)\le0$, and every $N_g$ is rigid.

Finally, we prove \eqref{sm:lin}. By \eqref{sm:target} the map
\[
 \Tmap[0]{N}{X}:\bigoplus_g\Hom_Q(N_g,P_g)\longrightarrow
 \Ext^1_Q(P_0,N_1)^*\oplus\Ext^1_Q(P_1,N_2)^*
\]
is defined, and its kernel $\Hom^0_\Pi(N,X)$ is one-dimensional. This map is surjective. Indeed, since $\incl$ is a
$\Pi$-map, the definition of $T$ gives $\Tmap{N}{X}(u\incl)(\xi)=\Tmap{X}{}(u)(\incl_*\xi)$
for $u\in\bigoplus_g\End_Q(P_g)$, so that $\Tmap[0]{N}{X}(u\incl)$ is the image of
$\Tmap[0]{X}{}(u)$ under the restriction maps
$\Ext^1_Q(P_g,P_{g+1})^*\to\Ext^1_Q(P_g,N_{g+1})^*$. These are surjective,
being dual to the maps $\Ext^1_Q(P_g,N_{g+1})\to\Ext^1_Q(P_g,P_{g+1})$,
which are injective since $\Hom_Q(P_g,P_{g+1}/N_{g+1})=0$: a quotient of
$P_{g+1}$ has the same starting vertices as $P_{g+1}$, namely $2,3$ for
$g=0$ and $4,5$ for $g=1$, and these exceed the starting vertices of
$P_g$, so that \eqref{eq:segment-hom} applies. As $\Tmap[0]{X}{}$ is surjective by
\cref{sm:end}, so is $\Tmap[0]{N}{X}$.

The groups $\Ext^1_Q(N_g,P_g)$ vanish because $P_g$ is rigid, $N_g$ is a
submodule, and $Q$ is hereditary, and $\Hom_Q(P_g,N_{g+1})=0$ by
\eqref{sm:hom-vanish}. Hence, by \eqref{eq:euler-Q},
\begin{align*}
 \dim\bigoplus_g\Hom_Q(N_g,P_g)&=\sum_g\langle \boldsymbol n_g,\boldsymbol p_g\rangle,\\
 \dim\bigl(\Ext^1_Q(P_0,N_1)^*\oplus\Ext^1_Q(P_1,N_2)^*\bigr)
 &=-\langle \boldsymbol p_0,\boldsymbol n_1\rangle-\langle \boldsymbol p_1,\boldsymbol n_2\rangle.
\end{align*}
As $\Tmap[0]{N}{X}$ is surjective with one-dimensional kernel, the first number
exceeds the second by one, which is \eqref{sm:lin}.
\end{proof}

We now express these conditions in terms of the multisegment of $N$. By
\cref{lem:sub-intervals}\cref{item:sub-start}, applied to $N_g\subset P_g$,
each $N_g$ is a direct sum of intervals taken from the disjoint lists
\begin{subequations}\label{sm:alphabets}
\begin{align}
 \Ilist_0&=([1,2],[2,2]),\label{sm:alph0}\\
 \Ilist_1&=([2,5],[3,5],[4,5],[5,5],[3,4],[4,4]),\label{sm:alph1}\\
 \Ilist_2&=([4,7],[5,7],[6,7],[7,7],[5,6],[6,6]).\label{sm:alph2}
\end{align}
\end{subequations}
We write $\Ilist=\Ilist_0\sqcup\Ilist_1\sqcup\Ilist_2$ for the set of the
fourteen intervals occurring in them, and
$\Ilist'=\Ilist\setminus\{[7,7]\}$. We write $m_{c,d}$ for the multiplicity of $[c,d]$ in $\m_N$, so that
$\boldsymbol n_g=\sum_{[c,d]\in\Ilist_g}m_{c,d}\gdim[c,d]$. By
\cref{lem:sub-intervals}\cref{item:sub-count} and \eqref{sm:layers}, the
number of summands of $N_g$ ending at $2,5,4,7,6$ is at most $3,3,8,8,8$
respectively, that is,
\begin{subequations}\label{sm:caps}
\begin{gather}
 m_{1,2}+m_{2,2}\le3,\label{sm:caps0}\\
 m_{2,5}+m_{3,5}+m_{4,5}+m_{5,5}\le3,\qquad m_{3,4}+m_{4,4}\le8,\label{sm:caps1}\\
 m_{4,7}+m_{5,7}+m_{6,7}+m_{7,7}\le8,\qquad m_{5,6}+m_{6,6}\le8,\label{sm:caps2}
\end{gather}
\end{subequations}
the three lines concerning $N_0$, $N_1$ and $N_2$.
Next we write \eqref{sm:q0} and \eqref{sm:lin} in terms of the
multiplicities. For intervals $\Gamma,\Gamma'$ put
$\langle\Gamma,\Gamma'\rangle=\langle\gdim\Gamma,\gdim\Gamma'\rangle_Q\in\{-1,0,1\}$ and
$m_\Gamma=m_{c,d}$ if $\Gamma=[c,d]$.
Expanding $\qf(\boldsymbol n_g)$ and $\langle \boldsymbol n_g,\boldsymbol n_{g+1}\rangle$, the identity
\eqref{sm:q0} translates to the quadratic relation
\begin{equation}\label{sm:q0-mult}
 \sum_{g=0}^2\sum_{\Gamma,\Gamma'\in\Ilist_g}
 \langle\Gamma,\Gamma'\rangle m_\Gamma m_{\Gamma'}
 +\sum_{g=0}^1\sum_{\substack{\Gamma\in\Ilist_g\\\Gamma'\in\Ilist_{g+1}}}
 \langle\Gamma,\Gamma'\rangle m_\Gamma m_{\Gamma'}=1.
\end{equation}
The left-hand side of \eqref{sm:lin} is additive in $\boldsymbol n_0,\boldsymbol n_1,\boldsymbol n_2$. For
$[c,d]\in\Ilist_g$ the contribution of one copy of $[c,d]$ is
\begin{equation}\label{sm:weight-def}
 w_{c,d}=\langle\gdim[c,d],\boldsymbol p_g\rangle+\langle \boldsymbol p_{g-1},\gdim[c,d]\rangle,
 \qquad \boldsymbol p_{-1}=0.
\end{equation}
A direct computation gives, in the order of \eqref{sm:alphabets},
\begin{equation}\label{sm:weights}
 \begin{gathered}
 (w_{1,2},w_{2,2})=(3,3),\qquad
 (w_{2,5},w_{3,5},w_{4,5},w_{5,5},w_{3,4},w_{4,4})=(0,8,11,3,5,8),\\
 (w_{4,7},w_{5,7},w_{6,7},w_{7,7},w_{5,6},w_{6,6})=(-3,5,13,8,-3,5).
 \end{gathered}
\end{equation}
Hence \eqref{sm:lin} reads
\begin{equation}\label{sm:weighted}
 \begin{aligned}
 &3(m_{1,2}+m_{2,2})+8m_{3,5}+11m_{4,5}+3m_{5,5}+5m_{3,4}+8m_{4,4}\\
 &\quad-3m_{4,7}+5m_{5,7}+13m_{6,7}+8m_{7,7}-3m_{5,6}+5m_{6,6}=1.
 \end{aligned}
\end{equation}

\begin{lemma}[Constraints on the multiplicities]\label{sm:alphabetlemma}
\leavevmode
\begin{enumerate}[label=\textup{(\roman*)},ref=\roman*]
\item\label{item:alph-middle}
One has $N_1\ne0$ and $m_{7,7}=0$.
\item\label{item:alph-rigid}
There are no $\Gamma,\Gamma'$ in the same list $\Ilist_g$ with
$m_\Gamma,m_{\Gamma'}>0$ and $\Gamma\prec\Gamma'$.
\end{enumerate}
\end{lemma}

\begin{proof}
For \cref{item:alph-middle}, recall that $N$ is indecomposable, being
cyclic. If $N_1=0$, then both reverse maps of $N$ vanish, so $N$ is an
indecomposable $Q$-module, hence a single interval $[c,d]$, and
\eqref{sm:lin} says $w_{c,d}=1$, which is excluded by \eqref{sm:weights}.
Thus $N_1\ne0$. Now let $S_7=[7,7]$ be an interval summand of $N_2$; it is a $Q$-summand of $N_Q$.
Since the intervals of $N_1$ end at $4$ or $5$, the reverse arrow $7\to6$ maps $S_7$ to $(N_1)_6=0$,
and no reverse arrow ends at the vertex $7$. Hence $S_7$ is a direct $\Pi$-summand of $N$, so that
$N=S_7$ by indecomposability, contradicting $N_1\ne0$.

Part \cref{item:alph-rigid} is \cref{sm:euler},
by \eqref{eq:prec-ext}.
\end{proof}

\begin{lemma}[Embedding inequalities]\label{sm:ineqlemma}
The multiplicities satisfy
\begin{equation}\label{sm:ineq}
 \begin{aligned}
 m_{2,5}&\le m_{1,2}, & m_{3,4}&\le m_{1,2}+m_{2,2}+5,\\
 m_{4,7}&\le m_{2,5}+m_{3,5}+5, & m_{4,7}&\le m_{3,5}+m_{3,4},\\
 m_{5,6}&\le m_{2,5}+m_{3,5}+m_{4,5}+5, & m_{5,6}&\le m_{3,5}+m_{4,5}+m_{3,4}+m_{4,4}.
 \end{aligned}
\end{equation}
\end{lemma}

\begin{proof}
Fix decompositions of $N_0,N_1,N_2$ into interval modules, and let
$r_v:X_{v+1}\to X_v$ be the reverse map of $X$ at the vertex $v$. Since $N$
is a submodule, $r_v$ maps $(N_{g+1})_{v+1}$ into $(N_g)_v$. For an
interval type $\Delta_i=[a,b]$ of $P_g$ and $v\in\Delta_i$, let
$\pr_i$ be the projection of $X_v$ onto the $\Delta_i$-space,
and let $\mu_i(v)$ be the number of summands $[c,d]$ of $N_g$ containing
$v$ with $a\le c\le b\le d$. By \eqref{eq:segment-hom} only these summands
have a nonzero component in $\Delta_i\otimes M_i$, so
$\dim\pr_i\bigl((N_g)_v\bigr)\le\mu_i(v)$.

Now let $(i,j)\in\Epairs$ with $v\in\Delta_i$ and $v+1\in\Delta_j$, and let
$S\subset(N_{g+1})_{v+1}$ be contained in the $\Delta_j$-space. Then
$G_{i,j}(S)=\pr_i r_v(S)\subset\pr_i\bigl((N_g)_v\bigr)$, whence
\begin{equation}\label{sm:ineq-principle}
 \dim S\le\dim\ker G_{i,j}+\mu_i(v),
\end{equation}
where $\dim\ker G_{i,j}$ is $0$ if $G_{i,j}$ is an identity and $5$ if it
is one of the $\rowmap_l$. The six inequalities \eqref{sm:ineq} are the instances
of \eqref{sm:ineq-principle} listed below; in each, $\dim S$ is the
left-hand side and $\dim\ker G_{i,j}+\mu_i(v)$ the right-hand side.
\begin{center}
\begin{tabular}{cccc}\toprule
$v$&$S$&$(i,j)$&$G_{i,j}$\\\midrule
$1$&$(N_1)_2$&$(1,2)$&$1_W$\\
$2$&span at $3$ of the summands $[3,4]$ of $N_1$&$(1,3)$&$\rowmap_1$\\
$3$&$(N_2)_4$&$(2,4)$&$\rowmap_2$\\
$3$&$(N_2)_4$&$(3,4)$&$1_V$\\
$4$&span at $5$ of the summands $[5,6]$ of $N_2$&$(2,5)$&$\rowmap_3$\\
$4$&span at $5$ of the summands $[5,6]$ of $N_2$&$(3,5)$&$1_V$\\\bottomrule
\end{tabular}
\end{center}
Here $(N_1)_2$ and $(N_2)_4$ lie in the $\Delta_2$- and $\Delta_4$-spaces,
since $[2,5]$ and $[4,7]$ are the only intervals of $\Ilist_1$ and
$\Ilist_2$ containing $2$ and $4$ respectively, and the summands $[3,4]$
and $[5,6]$ have no nonzero $Q$-maps to $\Delta_2$ and $\Delta_4$ by
\eqref{eq:segment-hom}.
\end{proof}

\subsection{The finite enumeration}\label{sm:finite-sec}
Recall that $\Ilist'$ is the set of the thirteen intervals of
\eqref{sm:alphabets} other than $[7,7]$.

\begin{lemma}[Exact finite exhaustion]\label{sm:finite}
Consider the vectors $(m_{c,d})_{[c,d]\in\Ilist'}$ of nonnegative
integers satisfying \eqref{sm:caps}, the condition of
\cref{sm:alphabetlemma}\cref{item:alph-rigid}, $m_{c,d}>0$ for some
$[c,d]\in\Ilist_1$, and the identities \eqref{sm:q0-mult} and
\eqref{sm:weighted}, with $m_{7,7}=0$. There are exactly
sixty-nine of them, and exactly one of them also satisfies
\eqref{sm:ineq}, namely
\begin{equation}\label{sm:fullcounts}
 m_{1,2}=m_{2,5}=3,\qquad m_{3,4}=m_{4,7}=m_{5,6}=8,\qquad
 m_{c,d}=0\text{ otherwise}.
\end{equation}
\end{lemma}

\begin{proof}
This is the only computer-assisted step. For $g=0,1,2$, list the vectors
$(m_{c,d})_{[c,d]\in\Ilist_g\cap\Ilist'}$ that satisfy \eqref{sm:caps0},
\eqref{sm:caps1} and \eqref{sm:caps2} respectively and the condition
of \cref{sm:alphabetlemma}\cref{item:alph-rigid}. Then form the triples of
such vectors with nonzero middle entry that satisfy \eqref{sm:weighted}.
Since the left-hand side of \eqref{sm:weighted} is the sum $w_0+w_1+w_2$ of
the contributions of the three vectors, it suffices to sort the vectors for
$g=2$ by $w_2$ and, for each pair of vectors for $g=0,1$, to look up those
with $w_2=1-w_0-w_1$. Finally, test \eqref{sm:q0-mult} and \eqref{sm:ineq} in
turn.
The counts are
\begin{center}
\begin{tabular}{lr}\toprule
Stage&Count\\\midrule
Vectors for $g=0$, including zero&10\\
Vectors for $g=1$, including zero&555\\
Vectors for $g=2$, including zero&3105\\
Triples satisfying \eqref{sm:weighted}&20249\\
Also satisfying \eqref{sm:q0-mult}&69\\
Also satisfying \eqref{sm:ineq}&1\\\bottomrule
\end{tabular}
\end{center}
\end{proof}

\begin{proof}[Conclusion of the proof of \cref{sm:main}]
By \cref{sm:end}, $X$ is indecomposable and rigid,
and its multisegment is $\m_X$ by \eqref{sm:layers}. Suppose that
$N$ is a \dist{} submodule of $X$, as in \cref{sm:cyclicity}. By \cref{sm:grading} we may assume that $N$ is
$\hgr$-graded. By \eqref{sm:alphabets} and \cref{sm:alphabetlemma}, the
multiplicities $m_{c,d}$ of $\m_N$ vanish outside $\Ilist'$ and satisfy
\eqref{sm:caps}, the condition of
\cref{sm:alphabetlemma}\cref{item:alph-rigid} and $N_1\ne0$. By
\cref{sm:euler} they satisfy \eqref{sm:q0} and \eqref{sm:weighted}, and by
\cref{sm:ineqlemma} they satisfy \eqref{sm:ineq}. Hence, by
\cref{sm:finite}, they are given by \eqref{sm:fullcounts}, that is,
$\m_N=\m_X$. So $\gdim N=\gdim X$, and $N=X$, a contradiction.
\end{proof}

\section*{AI disclosure}
During the preparation of this paper the authors used Claude (Anthropic), an AI assistant, to help with
drafting and revising the exposition, checking references against the literature, checking notation for
consistency, and editing the LaTeX source. The example in \cref{sec:thirty} was produced by ChatGPT (OpenAI).
All mathematical content was checked by the authors, who take full responsibility for the paper.


\def\cprime{$'$} 

\end{document}